\documentclass[a4paper]{article}
\usepackage[margin=1in]{geometry}
\usepackage[english]{babel}
\usepackage[utf8x]{inputenc}
\usepackage{amsmath,amssymb,amstext,amsthm}
\usepackage{graphicx}
\usepackage{url}
\usepackage[colorinlistoftodos]{todonotes}
\usepackage[colorlinks]{hyperref}

\newtheorem{theorem}{Theorem}[section]
\newtheorem{lemma}[theorem]{Lemma}
\newtheorem{corollary}[theorem]{Corollary}

\theoremstyle{definition}
\newtheorem{definition}[theorem]{Definition}

\newtheorem{remark}[theorem]{Remark}

\newcommand{\qeds}{\qed\vspace{.2cm}}

\newcommand{\xa}{x_{A}}
\newcommand{\xb}{x_{B}}
\newcommand{\ya}{y_{A}}
\newcommand{\yb}{y_{B}}
\newcommand{\ga}{g_{A}}
\newcommand{\gb}{g_{B}}
\newcommand{\ha}{h_{A}}
\newcommand{\hb}{h_{B}}
\newcommand{\la}{\langle}
\newcommand{\ra}{\rangle}
\newcommand{\R}{\mathbb{R}}

\newcommand{\K}{\mathcal{K}}
\newcommand{\thetak}{\vartheta^\K}

\newcommand{\cp}{\mathcal{CP}}
\newcommand{\cpsd}{{\mathcal{CS}_+}}
\newcommand{\dnn}{\mathcal{DNN}}
\newcommand{\psd}{{\mathcal{S}_+}}

\newcommand{\psdd}{{\mathcal{S}^d_+}}

\newcommand{\varthetam}{\vartheta^-}

\newcommand{\M}{M}

\newcommand{\A}{\ensuremath{\mathcal{A}}}
\newcommand{\B}{\ensuremath{\mathcal{B}}}
\newcommand{\C}{\ensuremath{\mathbb{C}}}

\newcommand{\hphi}{\hat{\phi}}
\newcommand{\comp}{\circ}
\newcommand{\schur}{\bullet}
\newcommand{\PA}{\hat{\A}}

\DeclareMathOperator{\tr}{Tr}

\DeclareMathOperator{\rel}{rel}
\DeclareMathOperator{\vect}{vec}
\DeclareMathOperator{\mat}{mat}

\DeclareMathOperator{\dar}{\downarrow}
\def\bi{\begin{itemize}}
\def\ei{\end{itemize}}
\newtheorem{result}{Result}

\def\be{\begin{equation}}
\def\ee{\end{equation}}

\title{Graph isomorphism: Physical resources, optimization models,  
and algebraic characterizations }
\author{Laura Man\v{c}inska\thanks{University of Copenhagen, mancinska@ku.dk} \and David E.~Roberson\thanks{Technical University of Denmark, dero@dtu.dk}  \and Antonios Varvitsiotis \thanks{Singapore University of Technology and Design, avarvits@gmail.com} }

\begin{document}
\maketitle

\begin{abstract}
In the  $(G,H)$-isomorphism game,  a verifier interacts with two non-communicating  players (called provers) by privately  sending each of them a random vertex from either $G$ or $H$, whose aim is  to convince the verifier that  two graphs $G$ and $H$ are isomorphic. 
%
%
%
In  recent work  along with Atserias, \v{S}\'{a}mal and Severini  [Journal of Combinatorial Theory, Series B, 136:89--328, 2019] we showed  that a verifier can be convinced that  two non-isomorphic graphs are   isomorphic, if the provers are allowed to  share quantum resources. 
 In this paper we  model    classical and quantum graph isomorphism  by  linear constraints  over certain complicated      convex cones, which we then relax  to a pair of tractable  convex models (semidefinite  programs). 
Our main result is a  complete algebraic characterization of the corresponding equivalence relations on graphs 
 in terms of appropriate  matrix algebras. 
Our techniques are an interesting mix of algebra, combinatorics, optimization,  and quantum~information.
\end{abstract}


\section{Introduction}
A pair of graphs $G$ and $H$ is isomorphic, denoted by $G\cong H$, if there exists a bijective map  from the vertex set of $G$ to the vertex set of $H$ that preserves adjacency and non-adjacency.  
 The problem of deciding whether two given graphs are isomorphic is of   fundamental
practical interest,  and at the same time, it plays    a central role in theoretical computer science as one of the few  problems in the class NP which is not known to be polynomial-time solvable or NP-complete.

Along with Atserias, \v{S}\'{a}mal, and Severini, the authors  recently introduced a non-local game that    captures the notion of graph isomorphism~\cite{qiso}. 
Specifically, in the $(G,H)$-graph isomorphism game there are two players, Alice and Bob, trying  to convince a third party, called the verifier,  that the graphs $G$ and $H$ are isomorphic. For this, the verifier randomly selects a pair of vertices $\xa, \xb \in V_G \cup V_H$ and sends to Alice and Bob respectively. After receiving their vertices, and without communicating, Alice and Bob  respond with vertices $\ya,\yb \in V_G \cup V_H$, where we  assume the vertex sets $V_G$ and $V_H$ are~disjoint.

The players win the game if the questions they were asked and the answers they provided 
indeed model an isomorphism from $G$ to $H$. Concretely, the first winning condition is that each player must respond with a vertex from the graph that the vertex they received was \emph{not} from, i.e.,
\begin{equation}\label{cond1}
\xa \in V_G \Leftrightarrow \ya \in V_H \text{ and } \xb \in V_G \Leftrightarrow \yb \in V_H.
\end{equation}
Furthermore,  setting  $\ga$ be the unique vertex of $G$ among $\xa$ and $\ya$, and  defining  $\gb, \ha$, and $\hb$ similarly,  the  second winning condition is that 
\begin{equation}\label{cond2}
\rel(\ga,\gb) = \rel(\ha,\hb),
\end{equation}
where $\rel(x,y)$ is the \emph{relationship} between vertices $x$ and $y$, i.e., whether they are equal, adjacent, or distinct non-adjacent.
Note that Equation \eqref{cond2} encodes   many constraints, e.g., if Alice and Bob are sent the same  vertices in $G$, then they must respond with the same  vertex of $H$ or if they are sent the endpoints of an edge of $G$ they need to respond with the endpoints of an edge of $H$. 
Furthermore, note that we do not explicitly require that $G$ and $H$ have the same number of vertices.

Alice and Bob  are allowed to agree  on a strategy before the start of the game, but are not allowed to communicate once the game has begun. This type of game is known as a \emph{nonlocal game}, since the players are usually thought of as being separated in space, which  prevents them from communicating after they receive their questions.  The parties only play one round of this game, and we only consider  strategies that win with certainty, i.e., with probability equal to one. We  refer to such  strategies as  \emph{perfect}.

It is easy to see that responding according to an isomorphism of $G$ and $H$ is a perfect strategy for the $(G,H)$-isomorphism game. Moreover, the converse also holds (see Section~\ref{subsec:classicalstrats}) and thus the isomorphism game characterizes the notion of isomorphism of graphs. Motivated by this, in the previous work \cite{qiso} we introduced the notions  quantum and non-signalling isomorphisms of graphs in terms of the existence of perfect quantum and non-signalling strategies for the graph isomorphism game. Furthermore,  we investigated these two relations, proving various necessary conditions for quantum isomorphism, giving a complete characterization of non-signalling isomorphism, and providing a method for constructing pairs of non-isomorphic graphs that are nevertheless quantum isomorphic.


In this work   we continue our study  of the graph isomorphism problem within the framework of nonlocal games. 
Our point of departure  is  a new equivalence relation on graphs, which is defined in terms of the feasibility of a certain linear conic program over an appropriate  convex cone. Specifically,  for any   convex cone of matrices $\K$ we say that graphs  $G$ and $H$  
are {\em $\K$-isomorphic}, and write $G \cong_\K H$, if there exists  a matrix $M$ with rows and columns indexed by $V_G \times V_H$ such that:
 \begin{align}
\sum_{h,h' \in V_H} M_{gh,g'h'} & = 1, \text{ for all } g,g' \in V_G \label{gsums}\\
\sum_{g,g' \in V_G} M_{gh,g'h'} & = 1, \text{ for all } h,h' \in V_H \label{hsums}\\
\phantom{\sum_{g,g' \in V_G}} M_{gh,g'h'} & = 0,   \text{ if } \rel(g,g') \ne \rel(h,h'), \label{zeros}\\
M & \in \K. \label{mink}
\end{align}
Any matrix satisfying \eqref{gsums}-\eqref{mink}
is called a \emph{$\K$-isomorphism matrix for $G$ to $H$}. 


Note that the entries in a $\K$-isomorphism matrix are not necessarily nonnegative, or even real, depending on the choice of cone $\K$. In this article  we study the graph equivalences  defined by  the notion of  $\K$-isomorphism for four  cones  of  matrices. 
The first one  is the cone of {\em positive semidefinite matrices (psd)}, denoted $\psd$, 
defined as the set of  Gram matrices of a set of vectors $v_1, \ldots, v_n$, i.e., $M_{ij} = v_i^Tv_j$. Second, we consider the \emph{doubly nonnegative cone}, denoted $\dnn$, which consists of entrywise-nonnegative psd  matrices. Third,  we consider  the  cone of  \emph{completely positive semidefinite} matrices \cite{laurent2015conic}, denoted $\cpsd$, which consists of  Gram matrices of psd matrices. Concretely, a  matrix  $M$ is completely positive semidefinite if there exist Hermitian  psd matrices  $\rho_1, \ldots, \rho_n$,  such that  $M_{ij} = \langle \rho_i, \rho_j \rangle := \tr(\rho_i^\dagger\rho_j)$. 
 Lastly, we consider  the \emph{completely positive cone}, denoted $\cp$, corresponding to  Gram matrix of entrywise nonnegative vectors. It is straightforward  to verify that
\begin{equation}
\cp \subseteq \cpsd \subseteq \dnn \subseteq \psd, \label{eq:cones}
\end{equation}
and   these containments are all strict for  matrices of size at least $5$~\cite{laurent2015conic}.

 All of these cones are of central importance to the field  mathematical optimization.  Most notably, linear optimation over the cone of psd matrices corresponds to semidefinite programming, an important family of optimization models with extensive modeling  power and efficient algorithms \cite{sdp}. Additionally, linear optimization over the completetely positive cone corresponds to completely positive  programming,
 a family of  optimization models that are  hard to solve but  have  significant expressive power \cite{burer}.


\paragraph{Summary of results and related work.}
In our first result we  expresses both classical and quantum graph isomorphism as  $\K$-isomorphism over appropriate cones of matrices. Specifically, we have that:

\begin{result}For any pair of graphs $G,H$  we have that $G$ and $H$ are isomorphic if and only  if $G\cong_\cp H$ and furthermore, $G$ and $H$ are quantum isomorphic if and only  if $G\cong_\cpsd H$. 
 \end{result}
The fact that quantum isomorphism is equivalent to the feasiblity of a linear conic program over the cpsd cone  is not surprising in view of  the  strong connections between  the cpsd cone and  the set of quantum correlations, e.g. see \cite{laurent2015conic, SV,cpsd2qc,conichomos}. On the other hand,  the formulation of graph isomorphism as a feasibility problem over the completely positive cone is to the best of our knowledge  new. A   related result is a  formulation for  GI  over the copositive cone \cite{gijben},  the dual of the completely positive~cone.

Furthermore, in~\cite{conichomos}, the notion of $\K$-homomorphism for various cones $\K$ was considered. These relations are related to homomorphisms in the same way that $\K$-isomorphisms are related to isomorphisms. In particular, $\cp$- and $\cpsd$-homomorphisms are equivalent to classical and quantum homomorphisms. 

As the problem of deciding  whether two graphs  are  classical (or  quantum)  isomorphic is  hard, it is important to identify   tractable necessary and/or sufficient conditions allowing to  checking this. 
In view of our first result,  we study the notion of  $\K$-isomorphism in the case of   the doubly nonnegative   and  positive semidefinite cones. Moreover, by the  chain of inclusions $\cp  \subseteq \cpsd \subseteq \dnn \subseteq \psd$, both $\dnn$- and $\psd$-isomorphism  are tractable relaxations of quantum  (and of classical) graph isomorphism. 
 The main contribution of this work  is  a complete algebraic characterization of  the  graphs that are  $\dnn$- and $\psd$-isomorphic respectively in terms of isomorphisms of appropriate matrix algebras.  

A linear  subspace of $\C^{n \times n}$  which is also closed under matrix multiplication is an \emph{algebra}.
A subalgebra  $\mathcal{A}$  of $\C^{n \times n}$ is called  \emph{coherent} if it is unital (i.e., contains the identity matrix), contains the  all-ones matrix, is closed under Schur product, and is  self-adjoint (i.e., closed under conjugate transpose). As the intersection of two coherent algebras is a coherent algebra we  can  define the \emph{coherent algebra of a graph $G$}, denoted by $\A_G$, as  the intersection of all coherent algebras containing the adjacency matrix of $G$. 


 \begin{result}
 Consider two graphs  $G, H$  with  adjacency matrices $A_G$ and $A_H$ and coherent algebras $\A_G$ and $\A_H$ respectively. Then, we have that    $G \cong_\dnn H$ if  and only if there exists an {isomorphism} between the  coherent algebras $\A_G$ and $\A_H$ that maps $A_G$ to $A_H$. 
\end{result}

As it turns out,  the notion of $\dnn$-isomorphism coincides with an  equivalence relation on graphs introduced  in 1968 by Weisfeiler and Leman~\cite{WL}, known today as the  2-dimensional \emph{Weisfeiler-Leman method}. In order to explain this link we first need to introduce some necessary background.

It is well-known that coherent algebras are in one-to-one correspondence  with coherent configurations. Indeed,  since  a coherent algebra $\A$ is closed under Schur product,  it must have an orthogonal (with respect to the Hilbert-Schmidt inner product) basis of 01-matrices, denoted by  $\{A_i : \ i \in \mathcal{I}\}$. Concretely, the 
   matrices  $A_i$  satisfy the following properties where $\schur$ denotes the \emph{Schur}, or entrywise, product: $(i) $   $A_i \schur A_j = \delta_{ij}A_i$, $(ii)$  $\sum_{i \in \mathcal{I}} A_i = J$, $(iii)$   $\sum_{i \in \Omega} A_i = I$ for some $\Omega \subseteq \mathcal{I}$, $(iv)$ 
 for each $i \in \mathcal{I}$, there exists a $j \in \mathcal{I}$ such that $A_i^\dagger = A_j$, and $(v) $  there exist numbers $p_{ij}^k$ for $i,j,k \in \mathcal{I}$, called the \emph{intersection numbers of $\A$},  such that $A_i A_j = \sum_k p_{ij}^k A_k$.


To  each  matrix  $A_i$  we  associate a subset of $V_G \times V_G$, namely the set of ordered pairs $(g,g')$ such that the $gg'$-entry of $A_i$ is 1. Equivalently, thinking of each such subset as a binary relation on $V_G$, i.e.,  $R_i=\{(g,g') \in V_G\times V_G: A_i(g,g')=1\}$,  properties $(i)$--$(v)$ imply  that that these relations form a  \emph{coherent configuration} \cite{hig}. Conversely, any coherent configuration corresponds to some coherent algebra.

Given a graph $G$, the 2-dimensional Weisfeiler-Lehman algorithm begins by labeling every ordered pair of vertices $(g,g')$ according to whether they are equal, adjacent, or non-adjacent.  At each step, for every ordered pair of vertices $(g_1,g_2)$, its label is augmented with  the $|V_G|$-element multiset of ordered pairs of labels of $(g_1,g),(g,g_2)$ for each $g \in V_G$. The algorithm terminates when the partition of $V_G \times V_G$ induced by the labels stabilizes. By specifying an ordering on the values of $\rel$, and ordering the labels lexicographically, the parts of the resulting partition inherit an isomorphism invariant ordering. If the partitions resulting from running this algorithm on two graphs are different (in a sense that can be made rigorous), then the graphs must be non-isomorphic. Otherwise, we say that the graphs are \emph{not distinguished by the Weisfeiler-Leman method}, which is an equivalence relation on graphs. As was shown in the original paper by Weisfeiler and Lehman, the resulting partition of $V_G \times V_G$ given by the above algorithm is exactly the coherent configuration corresponding to the coherent algebra of $G$ \cite{WL}. 



 In Section \ref{sec:characterization}  we characterize $\psd$-isomorphism by introducing an appropriate generalization of coherent algebras of graphs. Specifically,  we say that a subalgebra $\A$ of $\C^{n\times n}$ is {\em partially coherent} (with respect to  $ \{I, A_G\}$)
if it  is unital,  self-adjoint,  contains the all-ones matrix, and  is closed under Schur multiplication with the matrices  $I$ and $A_G$. 
As with coherent algebras, the intersection of two partially coherent algebras is again  a partially coherent algebra.
This allows to  define the \emph{partially coherent algebra of a graph $G$}, denoted $\PA_G$, to be the minimal  partially coherent algebra containing $A_G$.  We show the following:

 \begin{result}
 Consider two graphs  $G, H$  with  adjacency matrices $A_G$ and $A_H$ and partially coherent algebras $\PA_G$ and $\PA_H$ respectively. Then, we have that  $G \cong_\psd H$ if and only if  there exists a linear bijection $\phi: \PA_G \to \PA_H$ such that
 
 \begin{enumerate}
\item $\phi(M^\dagger) = \phi(M)^\dagger$ for all $M \in \PA_G$;
\item $\phi(MN) = \phi(M)\phi(N)$ for all $M,N \in \PA_G$;
\item $\phi(I) = I$, $\phi(A_G) = A_H$, and $\phi(J) = J$;
\item $\phi(M\schur N) = \phi(M)\schur\phi(N)$ for all $M \in \{I,A_G\}$ and $N \in \PA_G$.
\end{enumerate}
 \end{result}
 The notion of $\psd$-isomorphism appears to be a new graph relation, which we show implies several forms of cospectrality, e.g. see Lemmas \ref{lem:psd2cospec} and \ref{lem:laplacian}. Moreover, we show that $\psd$-isomorphism is equivalent to copsectrality of adjacency matrices when restricted to 1-walk-regular graphs (cf. Theorem \ref{thm:1walkregpsd}).

 Our main technique for studying  $\dnn$- and $\psd$-isomorphisms  is  a surprising  correspondence between $\K$-isomorphism matrices and linear maps $\Phi: \C^{V_G \times V_G} \to \C^{V_H \times V_H}$. Concretely, by considering a $\K$-isomorphism matrix $M$ for $G$ to $H$ as a Choi matrix, we can associate to it a linear map $\Phi_M: \C^{V_G \times V_G} \to~\C^{V_H \times V_H}$ given by $\left(\Phi_\M(X)\right)_{h,h'} = \sum_{g,g'} \M_{gh,g'h'} X_{g,g'}.$
 As it  turns out, maps constructed in this manner   have some remarkable properties. The idea for this construction is adopted  from Ortiz and Paulsen who applied it to winning correlations for the homomorphism game \cite{ortiz}.

As an  immediate consequence  of  the chain of inclusions  $\cp  \subseteq \cpsd \subseteq \dnn \subseteq \psd$, it follows that for all graphs $G$ and $H$ we have that 
\begin{equation}
G \cong H \ \Rightarrow \ G \cong_q H \ \Rightarrow \ G \cong_\dnn H \ \Rightarrow \ G \cong_\psd H,\label{eq:imps}
\end{equation}
 and in Section \ref{sec:separations} we show that none of these implications can be reversed.

In Section \ref{subsec:conicthetas}  we give yet another characterization of  $\K$-isomorphism  by combining  a conic generalization of   the celebrated  Lov\'asz theta   function and a new product of graphs. Specifically, for  any matrix  cone $\K$ consider  the graph parameter:
$$
\vartheta^\K(G) = \sup\left\{ \tr(MJ): \  M_{g,g'} = 0 \text{ if } g \sim g',  \  \tr(M) = 1, \ M \in \K\right\}.$$ 
For $\K = \psd$, the corresponding parameter is   the celebrated  Lov\'{a}sz theta function,  denoted by  $\vartheta$, whereas for  $\K = \dnn$, it is equal to a variant due to Schrijver, which  is usually denoted $\vartheta'$ \cite{lex}.
A nontrivial result~\cite{Pasechnik02} is that for $\K=\cp$ the parameter  $\vartheta^\cp$ is equal to the independence number of a graph.

In order to reformulate $\K$-isomorphism in terms of the graph parameter $\thetak$  we  make use of the  \emph{graph isomorphism product}, denoted $G \diamond H$, which has vertex set $V_G \times V_H$ and edges  $(g,h) \sim (g',h')$ if $\rel(g,g') \ne \rel(h,h')$.
In other words, vertices of $G \diamond H$ are adjacent exactly when the corresponding entry in an isomorphism matrix for $G$ to $H$ is required to be zero. Note that the isomorphism product of $G$ and $H$ is the complement of  the so-called weak {modular} product of graphs, e.g. see  \cite{handbook}.

 \begin{result}
 Consider two graphs  $G$ and $H$ and a matrix cone $\K \subseteq \psd$. Then $G \cong_\K H$ if and only if $\thetak(G \diamond H) = |V_G| = |V_H|$ and this value is attained. 
 \end{result}

For $\K = \cp$, the above result implies  that the graphs $G$ and $H$ are isomorphic if and only if $\alpha(G \diamond H) = {|V_G| = |V_H|}$. Nevertheless, this is not a new result  \cite{kozen}, and in fact it has long been known that $\alpha(G \diamond H)$ is the size of the largest common induced subgraph of $G$ and $H$.

Lastly, in Appendix \ref{sec:las} we show that $\dnn$- and $\psd$-isomorphisms respectively  correspond to the 
feasibility of (the first level of) the Lasserre hierarchy applied to appropriate  relaxations of the quadratic integer programming  formulation of the graph isomorphism problem.    
Specifically, two graphs $G$ and $H$    with  adjacency matrices $A_G$ and $A_H$ are isomorphic if and only if  there exists a permutation matrix $X=(X_{gh})$ such that $A_G=X^\top A_HX$. Thus,  the 0/1 solutions in the semi-algebraic set defined  by 
\begin{align}
&\label{new11} \sum_g X_{gh}=1,\\
&\sum_hX_{gh}=1,\\
&  \label{new31} X_{gh}X_{g'h'}=0 \text{ if } \rel(g,g')\ne \rel(h,h'),
\end{align}
encode all possible  isomorphisms between $G$ and $H$. 

Given  a semialgebraic set $\mathcal{K}=\{x\in [0,1]^n: g_i(x)\ge 0, \  i=1,\ldots, m\}$, the Lasserre hierarchy is a systematic method for producing tighter approximations to ${\rm conv}(\mathcal{K}\cap \{0,1\}^n)$ \cite{las}. Each level of the Lasserre  hierarchy  corresponds to a   semidefinite program and can be constructed using sums of squares representations of polynomials and the dual theory of moments. 

\begin{result}Two graphs  $G$ and $H$ are $\psd$-  (respectively $\dnn$-) isomorphic if the first level of the Lasserre hierarchy applied to \eqref{new11}-\eqref{new31} (respectively, adding entrywise nonnegativity) is feasible.
\end{result}

Summarizing, 
our results   give a surprising, and previously unknown, connection between coherent algebras, the Weisfeiler-Lehman algorithm, the Lasserre hierarchy, and Schrijver's theta function (see Theorem~\ref{thm:dnnsummary} and Remark~\ref{rem:dnnsummary}).

\section{Strategies for the isomorphism game}\label{sec:isogame}

In this section we recall the isomorphism game and briefly explain classical and quantum strategies. For more detailed background and additional properties  we refer the reader to~\cite{qiso}.

Recall that in the $(G,H)$-isomorphism game, each player receives a vertex from one of the graphs $G$ and $H$, and they must respond with a vertex from the other graph (see Equation~\ref{cond1}). In order to win, the vertices of $G$ that they receive/send must relate to each other (i.e., be equal, adjacent, or distinct and non-adjacent) in the same way as their vertices of $H$ (see Equation~\ref{cond2}). The players know $G$ and $H$ and can agree on a strategy beforehand, but they are not allowed to communicate once the game begins (i.e., once they receive their questions/vertices). We are interested in \emph{winning} or \emph{perfect} strategies, which are those that win with probability one.

\begin{remark}
Any winning strategy for the $(G,H)$-isomorphism game is also a winning strategy for the $(H,G)$-isomorphism game, as well as the $(\overline{G},\overline{H})$-isomorphism game. Here $\overline{G}$ refers to the \emph{complement} of $G$, i.e., the graph obtained from $G$ by replacing edges with non-edges and vice versa.
\end{remark}

For any fixed strategy for the $(G,H)$-isomorphism game, there is an associated joint conditional probability distribution, $p(\ya, \yb | \xa, \xb)$ for $\xa$, $\xb$, $\ya$, $\yb \in V_G \cup V_H$, which gives the probability of Alice and Bob responding with $\ya$ and $\yb$ when given inputs $\xa$ and $\xb$ respectively. The distribution $p$ is usually referred to as a \emph{correlation}. A given strategy for the $(G,H)$-isomorphism game is a winning strategy if and only if $p(\ya, \yb | \xa, \xb) = 0$ whenever $\xa, \xb, \ya$, and $\yb$ do not meet the winning conditions \eqref{cond1} and \eqref{cond2} defined  above.


\subsection{Classical Strategies}\label{subsec:classicalstrats}

In general, classical strategies allow Alice and Bob to have access to some shared randomness, such as a random binary string, which they can use to determine how they respond to the questions of the referee. However, for each value the shared randomness may assume, the corresponding strategy becomes deterministic. Mathematically, this says that any classical correlation $p$ can be written as $p = \sum_i \lambda_i p_i$ where $\lambda_i \ge 0$ for each $i$, with $\sum_i \lambda_i = 1$, and each $p_i$ corresponds to a deterministic classical strategy. The coefficients $\lambda_i$ encode the shared randomness used by the players. Since whether a correlation $p$ corresponds to a winning strategy is determined by its zeros, the correlation $p$ arises from a winning strategy if and only if $p_i$ is winning for all $i$ such that $\lambda_i > 0$.

A deterministic classical strategy consists of two functions $f_A$ and $f_B$ for Alice and Bob respectively that map inputs to outputs. Thus when Alice receives some input $x$, she will respond with $f_A(x)$, and Bob acts analogously. For the isomorphism game, it is not difficult to see that the functions $f_A$ and $f_B$ must be equal, and moreover that the restriction of them to $V_G$ (resp. $V_H$) is an isomorphism from $G$ to $H$ (resp. from $H$ to $G$). Furthermore, the restriction to $V_H$ is the inverse of the restriction to $V_G$. Thus the $(G,H)$-isomorphism game can be won perfectly with classical strategies if and only if $G$ and $H$ are actually isomorphic.

\subsection{Quantum Strategies}
In a quantum strategy the players can take advantage of shared quantum entanglement and measurements in order to produce their outputs. For our purposes we can restrict the shared entanglement to what are known as pure bipartite states of full Schmidt rank. Such a state corresponds to a unit vector $\psi \in \C^d \otimes \C^d$ which can be expressed as $\sum_{i\in[d]} \lambda_i \alpha_i \otimes \beta_i$ where $\mathcal{A} = \{\alpha_i : i\in[d]\}$ and $\mathcal{B} = \{\beta_i : i\in[d]\}$ are two orthonormal bases of $\C^d$ and $\lambda_i>0$ for all $i$. Any vector in $\C^d \otimes \C^d$ admits such a Schmidt decomposition if we allow $\lambda_i \ge 0$ and the restriction $\lambda_i>0$ reflects our assumption that the shared entangled state has full Schmidt rank. The two orthonormal bases, $\mathcal{A}$ and $\mathcal{B}$ are known as Schmidt bases of $\psi$ and there can be several choices for such a pair of bases. For a given shared state $\psi\in\C^d\otimes \C^d$, we intuitively think that the first tensor describes Alice's part of the state while the second one describes Bob's part. In order to extract classical information from this shared state Alice and Bob can measure their respective parts. A $k$-outcome quantum measurement of a $d$-dimensional system (space), also referred to as a POVM (Positive Operator Valued Measure), is described by a family of $k$ operators from $\psdd$ which add up to identity. We say that such a measurement is \emph{projective} if each of the positive semidefinite operators is an (orthogonal) projection. For more in-depth explanation of general quantum strategies we refer the  reader to \cite{NC}, or to \cite{qiso} for more details on quantum strategies for the isomorphism game specifically.

With these notions at hand we are ready to describe a general quantum strategy that Alice and Bob use to play a $(G,H)$-isomorphism game. First, Alice and Bob can choose the shared entangled state that they will use.
Next, each player can choose a quantum measurement which they will perform upon receiving $x\in V_G\cup V_H$.
Since any classical processing of the measurement outcome can be included in the measurement, without loss of generality, we can assume that each of the players respond with the measurement outcome they obtain. Hence, we index the measurement outcomes by elements of $V_G \cup V_H$. So a quantum strategy consists of a shared entangled state $\psi \in \mathbb{C}^d \otimes \mathbb{C}^d$ for some $d$, and quantum measurements $\mathcal{P}_x = (P_{xy} \in \psdd : y \in V_G \cup V_H)$ and $\mathcal{Q}_x = (Q_{xy} \in \psdd : y \in V_G \cup V_H)$ for each $x \in V_G \cup V_H$ for Alice and Bob respectively. According to the postulates of quantum mechanics, the probability of obtaining outcome $y$ and $y'$ upon measuring $\mathcal{P}_{x}$ and $\mathcal{Q}_{x'}$ respectively, is given by
\begin{equation}
p(y,y'| x,x') = \psi^\dagger \left(P_{xy} \otimes Q_{x'y'}\right)\psi.
\label{eqn:probp}
\end{equation}\\

It will often be useful to use the fact that the probability from Equation~\ref{eqn:probp} can be also expressed as
\begin{equation}\label{eqn:probs}
\psi^\dagger \left(P \otimes Q\right) \psi = \vect(\rho)^\dagger \left(P \otimes Q\right) \vect(\rho) = \vect(\rho)^\dagger \vect(P\rho Q^T) = \tr(\rho^\dagger P \rho Q^T)
\end{equation}
where $\psi := \vect(\rho)$ and $\vect : \C^{d_1 \times d_2} \to \C^{d_1} \otimes \C^{d_2}$ is the linear map defined by $\vect(e_i e_j^T) = e_i \otimes e_j$ and extended by linearity. In the above derivation, we have used the identities $\vect(AXB^T) = (A \otimes B) \vect(X)$ and $\tr(A^\dagger B) = \vect(A)^\dagger \vect(B)$ which can be verified by a direct calculation. We will also use the inverse map of $\vect$ which we denote by $\mat$. Note that $\mat$ takes vectors to matrices, \emph{i.e.}, $\mat : \C^{d_1} \otimes \C^{d_2} \to \C^{d_1 \times d_2}$. For notational convenience, we usually choose to express the shared state $\psi$, as well as the operators $P_{xy}$ and $Q_{xy}$, in a Schmidt basis of $\psi$. Note that in this basis, the operator $\rho  =  \mat(\psi)$ from Equation~(\ref{eqn:probs}) is a diagonal matrix with positive diagonal entries.

In~\cite{qiso} we showed that  perfect strategies of the isomorphism game  have a  special~form

\begin{theorem}\label{thm:stratform}
Let $G$ and $H$ be graphs and let $V := V_G \cup V_H$. Suppose that a shared state $\psi \in \C^d \otimes \C^d$ of full Schmidt rank and $\mathcal{P}_x = \{P_{xy} : y \in V\}$ and $\mathcal{Q}_x = \{Q_{xy} : y \in V\}$ for $x \in V$ comprise a winning strategy for the $(G,H)$-isomorphism game. Also, suppose that the operators $P_{xy}$ and $Q_{xy}$ as well as the shared state $\psi$ are expressed in a Schmidt basis of $\psi$ and let $\rho:= \mat(\psi)$. Then we have
\begin{enumerate}
\item $P_{xy} = Q_{xy}^T$ for all $x, y \in V$;
\item $P_{xy}$ and $Q_{xy}$ are projectors for all $x, y \in V$;
\item $P_{xy}\rho = \rho P_{xy}$ and $Q_{xy}\rho = \rho Q_{xy}$ for all $x, y \in V$;
\item $p(y, y' | x, x') := \psi^\dagger\left(P_{xy} \otimes Q_{x'y'}\right)\psi = 0$ if and only if $P_{xy}P_{x'y'} = 0$;
\item $P_{xy} = 0$ if $x,y \in V_G$ or $x,y \in V_H$;
\item $P_{xy} = P_{yx}$ for all $x, y \in V$.
\end{enumerate}
\end{theorem}

We also observed in \cite{qiso} that Theorem~\ref{thm:stratform} allows us to reformulate the existence of a quantum homomorphism in the following way.

\begin{theorem}\label{thm:qreform}
Let $G$ and $H$ be graphs. Then $G \cong_q H$ if and only if there exist projectors $P_{gh}$ for $g \in V_G$ and $h \in V_H$ such that
\begin{enumerate}
\item $\sum_{h \in V_H} P_{gh} = I$ for all $g \in V_G$;
\item $\sum_{g \in V_G} P_{gh} = I$ for all $h \in V_H$;
\item $P_{gh} P_{g'h'} = 0$ if $\rel(g,g') \ne \rel(h,h')$.
\end{enumerate}
\end{theorem}

Note that by items (1) and (6) of Theorem~\ref{thm:stratform}, any quantum correlation $p$ that wins the $(G,H)$-isomorphism game satisfies the following:
\[p(y,y'|x,x') = p(x,y'|y,x') = p(y,x'|x,y') = p(x,x'|y,y') \text{ for all } x,x',y,y' \in V_G \cup V_H.\]
In other words, switching the input and output, for Alice or Bob, does not effect the corresponding probability. We refer to any correlation with this property as being \emph{input-output symmetric}. This symmetry allows us to use a smaller matrix when formulating classical and quantum isomorphisms as conic feasibility problems. Note that since quantum winning correlations for the isomorphism game are input-output symmetric, so are classical correlations.

\section{Conic Formulations}\label{sec:conic}

In this section we will prove Result 1, that graphs $G$ and $H$ are isomorphic (resp.~quantum isomorphic) if and only if there is a $\cp$-isomorphism matrix (resp.~$\cpsd$-isomorphism matrix) for $G$ to $H$.

\subsection{Classical Correlations}

Suppose that $p$ is a correlation for the $(G,H)$-isomorphism game. Define the matrix $M^p$ with rows and columns indexed by $V_G \times V_H$ entrywise as:
\[M^p_{gh,g'h'} = p(h,h'|g,g').\]
Note that the matrix $M^p$ does not contain all of the probabilities of $p$, only those corresponding to inputs from $V_G$ and outputs from $V_H$. Thus, in general the matrix $M^p$ may not completely determine the correlation $p$. However, since winning classical or quantum correlations are input-output symmetric, such correlations $p$ are determined by the matrix $M^p$. Also note that since Alice and Bob are symmetric, i.e., $p(y,y'|x,x') =  p(y',y|x',x)$ for all $x,x',y,y' \in V_G \cup V_H$ for a winning classical or quantum correlation $p$, we have that $M^p$ is symmetric.

Recall from above that a matrix is completely positive if it is the Gram matrix of entrywise nonnegative vectors. Equivalently, a matrix $M$ is completely positive if $M = \sum_i p_ip_i^T$ where $p_i \in \mathbb{R}^d$ are entyrwise nonnegative vectors. The equivalence of these two definitions follows from the fact that the matrix $PP^T$ is the Gram matrix of the rows of $P$ but is also equal to $\sum_i p_ip_i^T$ where $p_i$ is the $i^\text{th}$ column of $P$. This formulation of completely positive matrices will be useful for proving Theorem~\ref{thm:cconic} below.

The proof of the following theorem resembles the proof of Theorem~4.2 from~\cite{conichomos},  a similar result for the homomorphism game. But there the only concern was showing that the \emph{existence} of a homomorphism was equivalent to the existence of a completely positive matrix satisfying certain properties.

\begin{theorem}\label{thm:cconic}
Suppose $G$ and $H$ are graphs and $p$ is a correlation for the $(G,H)$-isomorphism game. Then $p$ is winning classical correlation if and only if $p$ is input-output symmetric and $M^p$ is a $\cp$-isomorphism matrix.
\end{theorem}
\proof
Suppose that $p$ is  a winning classical correlation. Then we already know that $p$ is input-ouput symmetric, so it remains to show that $M^p$ is a $\cp$-isomorphism matrix. First, since $p$ is a winning correlation we have that Equations~(\ref{gsums}) and~(\ref{zeros}) hold. Also, since $p$ is input-output symmetric, we have that Equation~(\ref{hsums}) holds. Thus we only need to show that $M^p \in \cp$.

Since $p$ is classical, it is a convex combination of correlations arising from deterministic classical strategies. In other words, there are positive numbers $\lambda_1, \ldots, \lambda_k$ such that $\sum_{i=1}^k \lambda_i = 1$, and deterministic correlations $p_i$ such that
\[p = \sum_{i=1}^n \lambda_i p_i.\]
It is easy to see that $M^p = \sum_i \lambda_i M^{p_i}$. Since $p_i$ is deterministic, there exists a graph isomorphism $f_i :V_G \to V_H$ such that
\[p_i(h,h'|g,g') = \begin{cases} 1 & \text{if } h = f_i(g) \ \& \ h' = f_i(g') \\ 0 & \text{o.w.} \end{cases}\]
for all $g,g' \in V_G$ and $h,h' \in V_H$. Let $v^i$ be a real vector with coordinates indexed by $V_G \times V_H$ such that
\[v^i_{gh} = \begin{cases} 1 & \text{if } h = f_i(g) \\ 0 & \text{o.w.} \end{cases}.\]
It is straightforward to see that $M^{p_i} = v^i {v^i}^T$. Since the $v^i$ are entrywise nonnegative and the $\lambda_i$ are positive, it follows that $M^p$ is completely positive.

Conversely, suppose that $p$ is input-output symmetric and $M^p$ is a $\cp$-isomorphism matrix. It immediately follows from the definition of isomorphism matrices that $p$ is a winning correlation for the $(G,H)$-isomorphism game, so it remains to show that it is classical.

Since $M^p$ is completely positive, there are nonzero, entrywise nonnegative vectors $v^1, \ldots, v^k$ such that $M^p = \sum_i v^i {v^i}^T$. Let $N^i = v^i{v^i}^T$. Our aim is to show that each $v^i$ corresponds to a deterministic correlation, i.e., that $N^i$ is a scalar multiple of $M^{p_i}$ for some deterministic correlation $p_i$.

First, we show that the $N^i$ have uniform block sums, i.e., that $\sum_{h,h'} N^i_{gh,g'h'}$ does not depend on $g,g' \in V_G$. For each $i$, let $\widehat{N}^i$ be a matrix with rows and columns indexed by $V_G$ such that
\[\widehat{N}^i_{gg'} := \sum_{h,h' \in V_H} N^i_{gh,g'h'}.\]
Also let
\[\widehat{M}^p_{gg'} = \sum_{h,h' \in V_H} M^p_{gh,g'h'}.\]
So we have that $\widehat{M}^p = \sum_i \widehat{N}^i$. Note that $\widehat{M}^p$ can be obtained by conjugating $M^p$ by the matrix $I \otimes e$, where $e$ is the all ones vector, and $\widehat{N}^i$ can be obtained from $N^i$ similarly. Since $M^p$ and the $N^i$ are completely positive, they are also positive semidefinite and therefore so are $\widehat{M}^p$ and the $\widehat{N}^i$.

Since $\sum_{h,h'} p(h,h'|g,g') = 1$ for all $g,g' \in V_G$, we have that $\widehat{M}^p$ is equal to the all ones matrix $J$. However, $J$ is a rank 1 positive semidefinite matrix and we have shown it can be written as the sum of the positive semidefinite matrices $\widehat{N}^i$. This is only possible if each $\widehat{N}^i$ is a positive (since $v^i$ is nonzero) scalar multiple of $J$. Let $\mu_i$ be such that $\widehat{N}^i = \mu_i J$ and note that $\mu_i > 0$ for all $i$ and $\sum_i \mu_i = 1$.

Now we will show that each $v^i$ corresponds to a deterministic correlation. For notational simplicity, let $v = v^1$, $N = N^1$, and $\widehat{N} = \widehat{N}^1$. Since the $v^i$ are nonnegative, we must have that $N_{gh,g'h'} = 0$ if $\rel(g,g') \ne \rel(h,h')$. Suppose that $h \ne h'$ and $v_{gh} , v_{gh'} > 0$. Then $N_{gh,gh'} = v_{gh}v_{gh'} > 0$, a contradiction. Therefore, for each $g \in V_G$, there exists a unique $h \in V_H$ such that $v_{gh} > 0$. Let $f: V_G \to V_H$ be the function such that $v_{gf(g)} > 0$ for all $g \in V_G$. By an analogous argument to the above, we have that $f$ is injective. Moreover, since $N_{gh,g'h'} = 0$ for $\rel(g,g') \ne \rel(h,h')$, we must have that $f$ is an isomorphism of $G$ and $H$.

We want to show that $v_{gf(g)} = v_{g'f(g')}$ for all $g,g' \in V_G$. However,
\[\widehat{N}_{gg'} := \sum_{h,h' \in V_H} N_{gh,g'h'} = v_{gf(g)} v_{g'f(g')}.\]
Since $\widehat{N}$ is a positive multiple of $J$ by the above argument, we have that $v_{gf(g)} v_{g'f(g')}$ is constant for all $g,g' \in V_G$, and therefore $v_{gf(g)}$ is constant. It is easy to see that $v_{gf(g)} = \sqrt{\mu_1}$ and $N = \mu_1 M^q$ where $q$ is the correlation arising from the deterministic strategy obtained from using the isomorphism $f$. The same holds for all $i$ and therefore $p$ is a convex combination of correlations arising from deterministic classical strategies, i.e., $p$ is classical.\qeds

The above theorem is similar to a more general characterization of classical correlations given in~\cite{SV}. However, our result is different in that we use a smaller matrix whose rows and columns are indexed by $V_G \times V_H$, instead of $(V_G \cup V_H) \times (V_G \cup V_H)$. This is possible because of the input-output symmetry of winning classical/quantum correlations for the isomorphism game that is not present in arbitrary games.

Note that we must include the assumption of input-output symmetry in the theorem above, as otherwise we do not have any information about the other probabilities of $p$. Thus,  this assumption is dropped, it could be possible that $p$ is classical when the inputs are both from $V_G$, but not for other~inputs.


\subsection{Quantum Correlations}

To prove the characterization for winning quantum correlations for the isomorphism game, we first need the following lemma which was proven in~\cite{conichomos}  (see also \cite[Lemma 3.6]{SV}).

\begin{lemma}\label{lem:Gram}
Let $X$ and $Y$ be finite sets and let $M$ be the Gram matrix of vectors $w_{xy}$ for $x \in X$, $y \in Y$ which lie in some inner product space. The following assertions are equivalent:
\bi 
\item  Then there exists a unit vector $w$  satisfying $ \sum_{y \in Y} w_{xy} = w \text{ for all } x \in X$.
\item $ \sum_{y,y' \in Y} M_{xy,x'y'} = 1 \text{ for all } x,x' \in X.$ 

\ei 
\end{lemma}

Using this we can now prove the following:

\begin{theorem}\label{thm:qconic}
Suppose $G$ and $H$ are graphs and $p$ is a correlation for the $(G,H)$-isomorphism game. Then $p$ is a winning quantum correlation if and only if $p$ is input-output symmetric and $M^p$ is a $\cpsd$-isomorphism matrix.
\end{theorem}
\proof
Suppose that $p$ is a winning quantum correlation for the $(G,H)$-isomorphism game. By the same reasoning as in the proof of Theorem~\ref{thm:cconic}, we have that $M^p$ satisfies Equations~(\ref{gsums}), (\ref{hsums}), and~(\ref{zeros}). So it remains to show that $M^p \in \cpsd$.

Since $p$ is a quantum correlation, by Theorem~\ref{thm:stratform} and Equation~(\ref{eqn:probs}) we have that
\[M^p_{gh,g'h'} = \tr(\rho^\dagger P_{gh} \rho P_{g'h'}).\]
Moreover, we may assume that $\rho$ is a diagonal matrix with strictly positive diagonal entries. Thus $\rho$ is invertible and $\rho^{1/2}$ exists. Thus,
\[M^p_{gh,g'h'} = \tr(\rho P_{gh} \rho P_{g'h'}) = \tr\left((\rho^{1/2} P_{gh} \rho^{1/2}) (\rho^{1/2} P_{g'h'} \rho^{1/2})\right).\]
The expression on the righthand side above is the Hilbert-Schmidt inner product of the matrices $\rho^{1/2} P_{gh} \rho^{1/2}$ and $\rho^{1/2} P_{g'h'} \rho^{1/2}$, which are both positive semidefinite. Thus, $M^p$ is a Gram matrix of positive semidefinite matrices., i.e., it is completely positive semidefinite as desired.

Conversely, suppose that $p$ is input-output symmetric and $M^p$ is a $\cpsd$-isomorphism matrix. It immediately follows from the definition of isomorphism matrices that $p$ is a winning correlation for the $(G,H)$-isomorphism game. So it remains to show that $p$ is quantum. Since $M^p \in \cpsd$, there exist positive semidefinite matrices $\rho_{gh}$ for $g \in V_G$ and $h \in V_H$ such that $M^p_{gh,g'h'} = \tr(\rho_{gh}\rho_{g'h'})$. Since
\[\sum_{h,h'} M^p_{gh,g'h'} = 1 \text{ for all } g, g' \in V_G,\]
we can apply Lemma~\ref{lem:Gram}. Therefore, there exists a matrix $\rho$ such that
\[\sum_{h \in V_H} \rho_{gh} = \rho \text{ for all } g \in V_G,\]
and $\tr(\rho^\dagger\rho) = 1$. This implies that $\rho$ is positive semidefinite and that the column space of $\rho_{gh}$ is contained in the column space of $\rho$ for all $g \in V_G, h \in V_H$. Therefore, by restricting to a subspace if necessary, we may assume that $\rho$ is positive definite, and thus $\rho^{-1/2}$ exists.

Define $P_{gh} = \rho^{-1/2}\rho_{gh}\rho^{-1/2}$, and note that this is positive semidefinite. We have that
\[\sum_{h \in V_H} P_{gh} = \rho^{-1/2}\left(\sum_{h \in V_H} \rho_{gh}\right)\rho^{-1/2} = \rho^{-1/2}\rho\rho^{-1/2} = I.\]
Therefore, $\{P_{gh} : h \in V_H\}$ is a valid quantum measurement. We would like to also show that $\{P_{gh} : g \in V_G\}$ is a valid quantum measurement. To do this it suffices to show that $\sum_{g} \rho_{gh} = \rho$ for all $h \in V_H$. Since $M^p$ is an isomorphism matrix, we have that
\[\sum_{g,g'} M^p_{gh,g'h'} = 1 \text{ for all } h, h' \in V_H.\]
Thus we can apply Lemma~\ref{lem:Gram} again to obtain a positive semidefinite matrix $\rho'$ such that
\[\sum_{g \in V_G} \rho_{gh} = \rho' \text{ for all } h \in V_H.\]
However, we must have that
\[|V_G|\rho = \sum_{g \in V_G} \sum_{h \in V_H} \rho_{gh} = \sum_{h \in V_H} \sum_{g \in V_G} \rho_{gh} = |V_H|\rho'.\]
Since the sum of the entries of $M^p$ is equal to both $|V_G|^2$ and $|V_H|^2$ depending on which of Equations~(\ref{gsums}) and~(\ref{hsums}) you consider, we have that $\rho' = \rho$, and thus $\sum_{g} \rho_{gh} = \rho$ as desired. Therefore, $\{P_{gh} : g \in V_G\}$ is a valid quantum measurement.

Define $P_{hg} = P_{gh}$ and $P_{gg'} = 0 = P_{hh'}$ for all $g,g' \in V_G$ and $h,h' \in V_H$, and let $\psi = \vect(\rho)$. Then it is easy to see that for $g,g' \in V_G$ and $h, h' \in V_H$,
\[\psi^\dagger \left(P_{gh} \otimes P^T_{g'h'}\right)\psi = \tr(\rho P_{gh} \rho P_{g'h'}) = \tr(\rho_{gh}\rho_{g'h'}) = p(h,h'|g,g').\]
Switching $g$ and $h$ or $g'$ and $h'$ does not change any of the values above and so the correlation $p$ can be realized by Alice performing measurements $\mathcal{P}_x = (P_{xy} : y \in V_G \cup V_H)$ and Bob performing measurements $\mathcal{Q}_x = (P^T_{xy} : y \in V_G \cup V_H)$ on the shared state $\psi$. Thus $p$ is a quantum correlation and we have proven the theorem.\qeds

As with Theorem~\ref{thm:cconic}, the theorem above is similar to a more general characterization of synchronous quantum correlations given in~\cite{SV}, but we are able to use a smaller matrix due to input-output symmetry.

By the containments $\cp \subseteq \cpsd \subseteq \dnn \subseteq \psd$ and Theorems~\ref{thm:cconic} and~\ref{thm:qconic}, we have the following chain of implications:
\begin{equation}\label{eq:chain}
G \cong H \ \Rightarrow \ G \cong_q H \ \Rightarrow \ G \cong_\dnn H \ \Rightarrow \ G \cong_\psd H.
\end{equation}
We will see in Section~\ref{sec:separations} that none of these implications can be reversed.

\section{From isomorphism matrices to isomorphism maps}\label{sec:linearmaps}
Our main technique for studying  $\dnn$- and $\psd$-isomorphism is  a correspondence between isomorphism matrices and linear maps between the space of matrices indexed by $V_G$ and those indexed by $V_H$. 


\subsection{Linear  maps preserving matrix cones and their properties}\label{subsec:CPTP}

A linear map $\Phi: \mathbb{C}^{m \times m} \to \mathbb{C}^{n \times n}$ is  \emph{positive} if it maps psd matrices to psd matrices, i.e.,   $\Phi(X)$ is psd whenever $X$ is psd. 
We recall the following well-known theorem, e.g. see \cite{choi}: 
\begin{lemma}\label{choistuff}
For a  linear map $\Phi: \mathbb{C}^{m \times m} \to \mathbb{C}^{n \times n}$  the following assertions  are equivalent:
\begin{itemize}

\item[(i)] $\Phi$ is $m$-positive, i.e., the map $\text{id}_m\otimes \Phi$ is positive. 
\item[(ii)] The {\em Choi  matrix}  of $\Phi$, defined by $C_\Phi = \sum_{i,j=1}^m
 E_{ij} \otimes \Phi(E_{ij}),$ is psd.
\item[(iii)] $\Phi$ admits a {Kraus decomposition}, i.e.,  there exist matrices $K_i\in \C^{n\times m},  \ 1\le i\le mn$ such that 
$$\Phi(X)=\sum_iK_iXK_i^\dagger.$$
\item[(iv)] $\Phi$ is completely positive\footnote{Completely positive maps and completely positive matrices are not related (to our knowledge). This is simply an unfortunate collision of well-established terms. However, it should always be clear from context which notion we are referring to.}, i.e., the map  if $\text{id}_r \otimes \Phi$ is positive for all $r \in \mathbb{N}$. 

\end{itemize}
\end{lemma}
The  implications $(i)\implies (ii), (iii)\implies  (iv), (iv)\implies (i)$ are straighforward, whereas  $ (ii) \implies (iii)$ follows by considering a Cholesky factorization of the Choi matrix $C_\Phi$. Furthermore, we note that  is no relation between completely positive maps and the cone of completely positive matrices.


 A linear  map $\Phi: \mathbb{C}^{m \times m} \to \mathbb{C}^{n \times n}$ is  \emph{trace-preserving} (TP) if $\tr(\Phi(X)) = \tr(X)$ for all $X \in \mathbb{C}^{m \times m}$, and  {\em unital} if $\Phi(I_m) = I_n$. 
 Note that if $\Phi: \mathbb{C}^{m \times m} \to \mathbb{C}^{n \times n}$ is both trace-preserving and unital, we necessarily have that  $m = n$, since $n=\tr(I_n)=\tr(\Phi(I_m))=\tr(I_m)=m$. Completely positive trace-preserving (CPTP) linear maps are important in the theory of quantum information because they are in some sense the most general type of map allowed by the formalism of quantum mechanics.


In the next lemma we collect some useful properties of completely positive  maps in terms of their Kraus decompositions that   we invoke throughout this section. The proofs are easy and are omitted. 
\begin{lemma}\label{properties}
Consider a  completely positive  map   $\Phi: \mathbb{C}^{m \times m} \to \mathbb{C}^{n \times n} $  with Kraus decomposition $\Phi(X)=\sum_iK_iXK_i^\dagger$. Then, we have that: 
\begin{itemize} 
\item[(i)] The  adjoint of $\Phi $  is given by  $\Phi^\dagger(Y) = \sum_{i}  K_i^\dagger Y K_i$ for all $Y\in \C^{n\times n}$. In particular, the adjoint of a completely positive map is also completely positive (as it admits itself a Kraus decomposition).

\item[(ii)] $\Phi$   is  unital if and only if $\sum_i K_i K_i^\dagger=I$. 
\item[(iii)] $\Phi$ is trace-preserving if and only if $\sum_i K_i^\dagger K_i = I$. This is also equivalent to $\Phi^\dagger(I)=I$.
\item[(iv)] $\Phi(X^\dagger) = \Phi(X)^\dagger$ for any $X\in \C^{m\times m}$. 
\item[(v)] $\Phi$ (resp. $\Phi^\dagger$) is sum-preserving if $\Phi^\dagger(J)=J$ (resp. $\Phi(J)=J)$.

\end{itemize}
\end{lemma}
As  mentioned in  Lemma \ref{choistuff},  a linear map $\Phi: {\mathbb{C}^{m \times m} \to \mathbb{C}^{n \times n}}$ is completely positive  if and only if its Choi matrix is psd. It turns out that an analogous property holds for all   cones of interest to this work.  

Specifically,  
if $\Phi : \mathbb{C}^{m \times m} \to \mathbb{C}^{n \times n}$ is a linear map and $\mathcal{K}$ is a matrix cone,  we say that $\Phi$ is \emph{$\mathcal{K}$-preserving} if $X \in \mathcal{K}$ implies that  $\Phi(X) \in \mathcal{K}$. 
Using this  terminology, a linear map $\Phi$ is  positive if and only if  it is $\psd$-preserving. 
We now  prove a sufficient condition for showing a map is $\K$-preserving:

\begin{lemma}\label{dfvbrg}
For a matrix cone  $\mathcal{K} \in \{\cp, \cpsd, \dnn, \psd\}$ and   a  linear map $\Phi: \mathbb{C}^{m \times m} \to \mathbb{C}^{n \times n}$ we have that  $C_\Phi \in \K$ implies that  $\Phi$ is $\K$-preserving.
\end{lemma}
\begin{proof} 

Set  $C := C_\Phi \in \K$.  Using the fact that $X_{i,j} = (X \otimes J_n)_{i\ell,jk}$ 
 we have that
\be\label{vdfvb}
\Phi(X)_{\ell,k} = \sum_{i,j \in [m]} C_{i\ell, jk}X_{i,j}= \sum_{i,j \in [m]} C_{i\ell,jk} (X \otimes J_n)_{i\ell,jk} = \sum_{i,j \in [m]} (C \schur (X \otimes J_n))_{i\ell,jk}.
\ee
Note that the matrix $C \schur (X \otimes J_n)$ is a principal submatrix of $C \otimes (X \otimes J_n)$. Thus, since the all ones matrix $J$ is in $\K$ for all $\K \in \{\cp, \cpsd, \dnn, \psd\}$, and all such $\K$ are easily seen to be closed under tensor products and taking principal submatrices, we have that $C \schur (X \otimes J_n) \in \K$ whenever $X\in \K$. Furthermore,   equation  \eqref{vdfvb}  implies that the $\ell,k$  entry of $\Phi(X)$ is obtained by summing up the entries of the $\ell,k$ block of the block matrix  $C \schur (X \otimes J_n)$, and it is routine to check that $\K$ is closed under this operation for all $\K \in \{\cp, \cpsd, \dnn, \psd\}$. 
\end{proof}

We are now ready to prove the  following analogue of Lemma \ref{choistuff} for the cones of interest:

\begin{lemma}\label{lem:Kpreserving}
Consider a matrix cone  $\mathcal{K} \in \{\cp, \cpsd, \dnn, \psd\}$ and   a  linear map $\Phi: \mathbb{C}^{m \times m} \to \mathbb{C}^{n \times n}$. The  following assertions  are equivalent:
\begin{itemize}
\item[(i)] The map $\text{id}_m\otimes \Phi$ is $\mathcal{K}$-preserving. 
\item[(ii)]  The Choi matrix $C_\Phi$ lies in $  \K$.
\item[(iii)] $\Phi$ is completely $\mathcal{K}$-preserving, i.e.,  the map  $\text{id}_r \otimes \Phi$ is $\mathcal{K}$-preserving for all $r \in \mathbb{N}$.
\end{itemize}
\end{lemma}


\proof

$(i) \implies (ii)$.
The Choi matrix of $\Phi$ is equal to
\[C_\Phi = \sum_{i,j \in [m]} E_{ij} \otimes \Phi(E_{ij}) = (\text{id}_m \otimes \Phi)\left(\sum_{i,j} E_{ij} \otimes E_{ij}\right) = (\text{id}_m \otimes \Phi)\left(\sum_{i,j} e_ie_j^T \otimes e_ie_j^T\right) = (\text{id}_m \otimes \Phi)\left(\psi_m \psi_m^T\right),\]
where $\psi_m = \sum_{i=1}^m e_i \otimes e_i \in \mathbb{C}^m \otimes \mathbb{C}^m$. Since $\psi_m \psi_m^T \in \K$ for all $\K \in \{\cp, \cpsd, \dnn, \psd\}$, and 
$\text{id}_m\otimes \Phi$ is $\mathcal{K}$-preserving,  we have that $C_\Phi \in \K$.


$(ii)\implies (iii)$. 
Assume  that $C := C_\Phi \in \K$.  By Lemma \ref{dfvbrg} it suffices to  show that the Choi matrix of $\mathrm{id}_r \otimes \Phi$ is in $\K$ for all $r$. 
Using $\text{id}_{rm} = \text{id}_r \otimes \text{id}_m$, we have that the Choi matrix of $\text{id}_r \otimes \Phi$ is equal to
\begin{align*}
(\text{id}_r \otimes \text{id}_m) \otimes (\text{id}_r \otimes \Phi) \left(\sum_{i,j \in [r], \ \ell,k \in [m]} (E_{ij} \otimes E_{\ell k}) \otimes (E_{ij} \otimes E_{\ell k}) \right) &= \sum_{i,j \in [r], \ \ell,k \in [m]} E_{ij} \otimes E_{\ell k} \otimes E_{ij} \otimes \Phi(E_{\ell k}).
\end{align*}
Up to a permutation of the tensors, this is equal to
\[\left(\sum_{i,j \in [r]} E_{ij} \otimes E_{ij}\right) \otimes C_\Phi = \psi_r\psi_r^T \otimes C_\Phi,\]
which is in $\K$ since $\K$ is closed under tensor products. Since permuting the tensors is equivalent to conjugation by some permutation matrix, i.e., consistently relabeling rows and columns, this does not change whether the matrix is in $\K$ and thus we have proven the claim.

$(iii)\implies (i)$. Straightforward. \qeds 

We conclude this section with a lemma we use repeatedly in the remainder of this  article. 

\begin{lemma}\label{newlemma2}
Let $D \in \mathbb{C}^{m \times n}$ be a matrix and let $u \in \mathbb{C}^n$, $v \in \mathbb{C}^m$. Then the following are equivalent:
\begin{enumerate}
\item[$(1)$] $D(u \schur w) = v \schur (Dw)$ for all $w \in \mathbb{C}^n$.
\item[$(2)$] $D_{ij} = 0$ whenever $v_{i} \ne u_{j}$.
\item[$(3)$] $D^\dag(v \schur z) = u \schur (D^\dag z)$ for all $z \in \mathbb{C}^m$.
\end{enumerate}
\end{lemma}
\begin{proof}
Consider the linear maps $f(w) = D(u \schur w)$ and $g(w) = v \schur (Dw)$. Letting $D_j$ be the $j^\text{th}$ column of $D$ and $e_j$ the $j^\text{th}$ standard basis vector, it follows that $f(e_j) = u_jD_j $ and $g(e_j) = v \schur D_j$. Consequently, $D(u \schur w) = v \schur (Dw)$ for all  $w\in \C^d$  is equivalent to the statement
\[ f(e_j) = g(e_j) \text{ for all } j \in [d] \iff u_jD_j = v \schur D_j \text { for all } j \in [d],\]
which is in turn equivalent to $D_{ij} = 0$ whenever $u_j \ne v_i$. Lastly, to get the third equivalence, note that $D_{ij} =0$ whenever $v_i \ne u_j$ is equivalent to $D^\dag_{ji}=0$ whenever $u_j \ne v_i$.
\end{proof}

Now consider a map $\Phi : \mathbb{C}^{V_G \times V_G} \to \mathbb{C}^{V_H \times V_H}$ with Choi matrix $M$. Define $\tilde{M}$ entrywise as $\tilde{M}_{hh',gg'} = M_{gh,g'h'}$. It is straightforward to check that $\vect(\Phi(X)) = \tilde{M}\vect(X)$ for all $X \in \mathbb{C}^{V_G \times V_G}$ and $\vect(\Phi^\dag(Y)) = \tilde{M}^\dag \vect(Y)$ for all $Y \in \mathbb{C}^{V_H \times V_H}$. Through this correspondence Lemma~\ref{newlemma2} immediately yields the following:

\begin{lemma}\label{newlemma}
Let $\Phi : \mathbb{C}^{V_G \times V_G} \to \mathbb{C}^{V_H  \times V_H}$ be a linear map with  Choi matrix $M$. For any fixed pair of matrices  $X \in \mathbb{C}^{V_G \times V_G}$ and $Y \in \mathbb{C}^{V_H \times V_H}$ the  following are equivalent:
\begin{enumerate}
\item[$(1)$] $\Phi(X \schur W) = Y \schur \Phi(W)$ for all $W \in \mathbb{C}^{V_G \times V_G}$.
\item[$(2)$] $M_{gh,g'h'} = 0$ whenever $X_{gg'} \ne Y_{hh'}$.
\item[$(3)$] $\Phi^\dag(Y \schur Z) = X \schur \Phi^\dag(Z)$ for all $Z \in \mathbb{C}^{V_H \times V_H}$.
\end{enumerate}
\end{lemma}

\subsection{The Choi matrix as a  $\K$-isomorphism  matrix }\label{subsec:isomaps}

\begin{theorem}\label{thm:isomap}
Consider two graphs $G,H$, a cone of matrices  $\K \in \{\cp, \cpsd, \dnn, \psd\}$ and  a linear map $\Phi: \mathbb{C}^{V_G \times V_G} \to~\mathbb{C}^{V_H \times V_H}$. The following assertions are equivalent:
\bi 
\item[$(1)$]
The Choi matrix of $\Phi$ is a $\K$-isomorphism matrix from $G$ to $H$. 
\item[$(2)$] $\Phi$  is a $\K$-isomorphism map from $G$ to $H$, i.e., it satisfies
 \begin{align}
 &  \Phi \text{ is completely } \K\textnormal{-preserving}  \label{eq:preserving}\\
&\Phi(I \schur X) = I \schur \Phi(X) \text{ for all } X \in \mathbb{C}^{V_G \times V_G} \label{eq:verts}\\
&\Phi(A_G \schur X) = A_H \schur \Phi(X) \text{ for all } X \in \mathbb{C}^{V_G \times V_G} \label{eq:edges}\\
&\Phi(J) = J = \Phi^\dagger(J)\label{eq:J},
\end{align}
\item[$(3)$] $\Phi^\dag$ is a $\K$-isomorphism map from  $H$ to $G$. 
\ei 
\end{theorem}
\begin{proof}$(1)\iff (2)$. 
Let $\Phi: \mathbb{C}^{V_G \times V_G} \to \mathbb{C}^{V_H \times V_H}$ be a linear map and $M$ its Choi matrix.  Consider the matrix $\tilde{M}$ with columns indexed by $V_G \times V_G$ and rows indexed by $V_H \times V_H$ defined as $\tilde{M}_{hh',gg'} = M_{gh,g'h'}.$
By Lemma~\ref{lem:Kpreserving}, we have that $M \in \K$ if and only if $\Phi$ is completely $\K$-preserving, i.e.,
\begin{equation}\label{Kpres}
\eqref{mink} \Leftrightarrow \eqref{eq:preserving}.
\end{equation}
Also, Conditions~\eqref{gsums} and~\eqref{hsums} are equivalent to $\tilde{M}$ having all row and column sums equal to 1 respectively, which in turn are equivalent to $\Phi(J) = J$ and $\Phi^\dag(J) = J$ respectively. Thus
\begin{equation}\label{Jsums}
\eqref{gsums} \ \& \ \eqref{hsums} \Leftrightarrow \eqref{eq:J}.
\end{equation}

Lastly, we  show that Condition~\eqref{zeros} holding for $M$ is equivalent to Conditions~\eqref{eq:verts} and  \eqref{eq:edges} holding for $\Phi$. Indeed, by Lemma \ref{newlemma}  we have that 
$ \Phi(A_G \schur X) = A_H \schur \Phi(X) \text{ for all } X \in \mathbb{C}^{V_G \times V_G}$ is equivalent to 
$M_{gh,g'h'} = 0$ whenever $(A_G)_{gg'} \ne (A_H)_{hh'}$, i.e., whenever ($g \sim g'$ \& $h \not\sim h'$) or ($g \not\sim g'$ \& $h \sim h'$). Similarly Lemma \ref{newlemma}  implies  that $\Phi(I \schur X) = I \schur \Phi(X)$ is equivalent to $M_{gh,g'h'} = 0$ whenever ($g = g'$ \& $h \ne h'$) or ($g \ne g'$ \& $h = h'$).
Summarizing, we have that Conditions~\eqref{eq:verts} and  \eqref{eq:edges} holding for $\Phi$ is equivalent to $M_{gh,g'h'} = 0$ whenever $\rel(g,g') \ne \rel(h,h')$, i.e.,
\begin{equation}\label{schurzeros}
\big(\eqref{eq:verts} \ \& \ \eqref{eq:edges}\big) \Leftrightarrow \eqref{zeros}.
\end{equation}
Combining the equivalences in~\eqref{Kpres}, \eqref{Jsums}, and~\eqref{schurzeros} yields the theorem.

$(2) \iff (3)$. Follows immediately by Lemma \ref{newlemma}. 
\end{proof}

\begin{remark}\label{isomapsremark}
We conclude this section by listing some further useful properties satisfied by isomorphism maps.
First,
note that~(\ref{eq:verts}) and~(\ref{eq:edges}) further imply that
\begin{equation}
\Phi(A_{\overline{G}} \schur X) = A_{\overline{H}} \schur \Phi(X) \text{ for all } X \in \mathbb{C}^{V_G \times V_G}, \label{eq:nonedges}
\end{equation}
since $A_{\overline{G}} = J - I - A_G$ and $J$ is the identity with respect to the   Schur product. 
Furthermore,  since $\Phi(J)=J=\Phi^\dag(J)$ it follows respectively by \eqref{eq:verts} and \eqref{eq:edges} that: 
$$\Phi(I)=I \text{ and } \Phi(A_G)=A_H.$$
Lastly, the fact that $\Phi$ is sum-preserving combined with $\Phi(I)=I$ shows that $G$ and $H$ have the same number of vertices. Analogously,  $\Phi(A_G)=A_H$ implies that $G$ and $H$ have the same number of edges. 
\end{remark}

\subsection{Some additional properties of isomorphism maps}

\begin{lemma}\label{lem:XYcommute} 
Consider a linear map  $\Phi: \C^{n\times n}\to \C^{n\times n} $ which is    completely positive, trace-preserving, and unital.  Then, for any pair of   matrices $X,Y$ such that $\Phi(X)=Y$ and $\Phi^\dag(Y)=X$ we have that
$$\Phi(XW) = \Phi(X)\Phi(W) \ \text{ and } \ \Phi(WX) = \Phi(W)\Phi(X), \text{ for all  matrices } W.$$
\end{lemma}

\begin{proof}The presented proof is a small modification of the arguments in \cite{watrous}. 
As $\Phi$ is completely positive it admits a 
  a Kraus decomposition  $\Phi(Z)=\sum_{i=1}^mK_i ZK_i^\dagger$. The crux of the proof is to  show that 
\be\label{cvevv}
K_i X = YK_i\  \text{ and }\  XK_i^\dagger = K_i^\dagger Y \text{ for all }   i \in [m].
\ee 
For this, set  $\mathcal{Z} = \sum_i (K_i X - YK_i)(K_i X - YK_i)^\dagger$ we have 
\begin{align*}
\mathcal{Z} &= \sum_i (K_i X - YK_i)(X^\dagger K_i^\dagger - K_i^\dagger Y^\dagger) \\
&= \sum_i  K_i X X^\dagger K_i^\dagger -  \sum_i K_i X K_i^\dagger Y^\dagger - \sum_i YK_i X^\dagger K_i^\dagger +  \sum_iYK_i K_i^\dagger Y^\dagger \\
&=  \Phi(XX^\dagger) - \Phi(X) Y^\dagger - Y\Phi(X^\dagger)  + Y Y^\dagger  \\
&= \Phi(XX^\dagger) - YY^\dagger,
\end{align*}
 where to get the last equality we used the assumption  $\Phi(X)=Y$ and that $\Phi(X^\dagger)=Y^\dagger$, the latter  following  by  Lemma~\ref{properties} $(iv)$. 
 As $\mathcal{Z}$ is psd (since it is the sum of psd matrices) 
 we have that 
 \be\label{scsdvf}
 0 \le \tr(\mathcal{Z}) = \tr\left(\Phi(XX^\dagger) - YY^\dagger\right)=\tr(XX^\dagger - YY^\dagger),
 \ee
 where for the last equality we used that $\tr(\Phi(XX^\dagger)) =\tr(XX^\dagger) $ since  $\Phi$ is trace preserving. 
 
 By assumption we also have that $\Phi^\dag(Y) = X$, and $\Phi^\dag$  is trace-preserving as $\Phi$ is assumed to be unital.  In a similar manner as above we  get that  
  \be\label{ererg}
  \tr(\mathcal{Z})=\tr(YY^\dagger - XX^\dagger) \ge 0.
  \ee Combining \eqref{scsdvf} and \eqref{ererg} we get  that $\tr(\mathcal{Z}) = \tr(XX^\dagger - YY^\dagger) = 0$, and as  $\mathcal{Z}$ is psd, this further implies   that $\mathcal{Z} = 0$. As $\mathcal{Z}$ is the sum of psd matrices,  every term in the sum $\mathcal{Z}=\sum_i (K_i X - YK_i)(K_i X - YK_i)^\dagger$ is equal to zero, which  in turn implies that $K_i X - YK_i = 0$ for all $i$, i.e., that $K_i X = YK_i$.

Lastly, using that   $K_i X = YK_i$ for all $i$ we have that for any matrix $W$:
\[\Phi(XW) = \sum_i K_i XW K_i^\dagger = \sum_i YK_i W K_i^\dagger = Y\Phi(W).\]

Lastly, repeating the above argument  with the  matrix $\sum_i (X K_i^\dagger - K_i^\dagger Y)^\dagger(X K_i^\dagger - K_i^\dagger Y)$ it follows  that $XK_i^\dagger = K_i^\dagger Y$ for all $i$, and thus,  $\Phi(WX) = \Phi(W)Y$. 
\end{proof}

\begin{lemma}\label{lem:cospectral}
Consider a linear map  $\Phi: \C^{n\times n}\to \C^{n\times n} $ which is   completely positive, trace-preserving, and unital. Then, for any pair of   Hermitian matrices $X,Y$ such that $\Phi(X)=Y$ and $\Phi^\dag(Y)=X$ we have that $X$ and $Y$ 
are cospectral, and furthermore,  if $E_\lambda$ and $F_\lambda$ are projections onto the $\lambda$-eigenspaces of $X$ and $Y$ respectively, then $\Phi(E_\lambda) = F_\lambda$ and $\Phi^\dag(F_\lambda) = E_\lambda$.
\end{lemma}
\proof
By Lemma~\ref{lem:XYcommute}, we have that $\Phi(XW) = Y\Phi(W)$ for any matrix $W$ and therefore $\Phi(f(X)) = f(Y)$ for any polynomial $f$. Let $\lambda$ be an eigenvalue for $X$ and let $E_\lambda$ be the corresponding orthogonal projector. 
Then, we have that 
\be\label{dvadvfgb}
\Phi(E_\lambda)= \Phi(E_\lambda^2)=\Phi(E_\lambda)^2,
\ee
where the first equality follows as $E_\lambda$ is a projector and the second one as 
$E_\lambda$ can be written as a polynomial in $X$, concretely $E_\lambda=\prod_{\tau\ne \lambda}{(X-\tau I)\over \lambda -\tau}$. 
Consequently,   $\Phi(E_\lambda)$ is an orthogonal projector since it is idempotent by   \eqref{dvadvfgb}, and  Hermitian since $\Phi(E_\lambda)^\dag=\Phi(E_\lambda^\dag)=\Phi(E_\lambda)$.  Furthermore, since $\Phi$ is trace-preserving and the rank of a projector is equal to its trace, the rank of $\Phi(E_\lambda)$ is equal to that of $E_\lambda$.~Furthermore, 
\[Y\Phi(E_\lambda) =\Phi(X)\Phi(E_\lambda)=  \Phi(X E_\lambda) = \Phi(\lambda E_\lambda) = \lambda \Phi(E_\lambda),\]
which means that the range of $\Phi(E_\lambda)$ is  contained on the $\lambda$-eigenspace of $Y$. Summarizing we showed that if $\lambda$ is an eigenvalue for $X$ then it also an eigenvalue for $Y$ and furthermore ${\rm mult}(Y,\lambda)\ge {\rm mult}(X,\lambda)$.  The symmetric argument shows that $X$ and $Y$ have the same multiset of eigenvalues, i.e., they are cospectral. 
Lastly, combining the inclusion  ${\rm range}(\Phi(E_\lambda))\subseteq {\rm Ker}(Y-\lambda I)$ with the fact that both  subspaces have the same dimension, it follows that $\Phi(E_\lambda) = F_\lambda$ and similarly $\Phi^\dagger(F_\lambda) = E_ \lambda$.\qeds

\section{Characterizing  $\psd $-isomorphic graphs}\label{sec:characterization}
\subsection{Partially coherent algebras} 

Suppose that $S$ is some subset of $\C^{n \times n}$. We say that an algebra $\A$ is an \emph{$S$-partially coherent algebra} if $\A$
\begin{enumerate}
\item is unital;
\item is self-adjoint;
\item contains the all ones matrix;
\item is closed under Schur multiplication by any matrix in $S$.
\end{enumerate}

Note that the last two properties above imply that any $S$-partially coherent algebra must contain every element of $S$. On the other hand, any coherent algebra containing $S$ is $S$-partially coherent. The smallest example that we know of an $S$-partially coherent algebra that is not a coherent algebra is the algebra of polynomials in the adjacency matrix of the Hoffman graph. We have verified by computer that this algebra is not a coherent algebra, but is an $S$-partially coherent algebra for $S = \{I,A\}$ where $A$ is the adjacency matrix of the Hoffman graph which we will see in Section~\ref{sec:separations}.

As with coherent algebras, it is easy to see that the intersection of two $S$-partially coherent algebras is an $S$-partially coherent algebra. Therefore, there is some minimal $S$-partially coherent algebra for any $S$. This will be equal to the set of matrices that can be expressed using the elements of $S \cup \{I,J\}$ and a finite number of the operations of addition, scalar multiplication, matrix multiplication, conjugate transposition, and Schur multiplication where at least one of the factors is an element of $S$. 

We define the \emph{partially coherent algebra of a graph $G$}, denoted $\PA_G$, to be the minimal $S$-partially coherent algebra where $S = \{I, A_G\}$. Note that this will also be $S'$-partially coherent for $S' = \{I, A_G, A_{\overline{G}}\}$ since $A_{\overline{G}} = J - I - A_G$ and $J$ is the Schur identity. 

\begin{definition}
Let $G$ and $H$ be graphs with adjacency matrices $A_G$ and $A_H$ and partially coherent algebras $\PA_G$ and $\PA_H$ respectively. We say that $G$ and $H$ are \emph{partially equivalent} if there exists an linear bijection $\phi: \PA_G \to \PA_H$ such that
\begin{enumerate}
\item $\phi(M^\dagger) = \phi(M)^\dagger$ for all $M \in \PA_G$;
\item $\phi(MN) = \phi(M)\phi(N)$ for all $M,N \in \PA_G$;
\item $\phi(I) = I$, $\phi(A_G) = A_H$, and $\phi(J) = J$;
\item $\phi(M\schur N) = \phi(M)\schur\phi(N)$ for all $M \in \{I,A_G\}$ and $N \in \PA_G$.
\end{enumerate}
We refer to $\phi$ as a \emph{partial equivalence} of $G$ and $H$.
\end{definition}

Note that the conditions above imply that $\phi(A_{\overline{G}}) = A_{\overline{H}}$ where $A_{\overline{G}}$ and $A_{\overline{H}}$ are the adjacency matrices of the complements of $G$ and $H$ respectively. Furthermore, they also imply that $\phi(A_{\overline{G}} \schur N) = A_{\overline{H}} \schur \phi(N)$ for all $N \in \PA_G$. Note that if it exists, a partial equivalence $\phi$ of graphs $G$ and $H$ is completely determined since $\phi(A_G) = A_H$.

If $\phi$ is an equivalence of graphs $G$ and $H$, then any function of $I, A_G$, and $J$ using the operations of addition, scalar multiplication, matrix multiplication, entrywise multiplication, and conjugate transposition is mapped to the same function with $A_G$ replaced by $A_H$. This obviously still holds if we restrict to functions in which entrywise multiplication can only be used when one of the factors is $I$ or $A_G$. Since the space of matrices that can be written as such functions is exactly the partially coherent algebra of $G$, we have that the restriction of the equivalence $\phi$ to $\PA_G$ is a partial equivalence of $G$ and $H$. Thus, any pair of equivalent graphs are also partially equivalent, as one would expect.

\subsection{The characterization}

\begin{theorem}\label{thm:psdiso}
Two graphs $G$ and $H$ are partially equivalent if and only if   $G \cong_\psd H$. 
\end{theorem}
\proof

Assume  that $G \cong_\psd H$ and  let $M$ 
be a $\psd$-isomorphism matrix.  Let $\Phi:~\C^{V_G \times V_G} \to \C^{V_H \times V_H}$ be the linear map whose Choi matrix is $M$
As $M$  is a $\psd$-isomorphism matrix it follows by 
 Theorem \ref{thm:isomap}   that $\Phi$ is  a $\psd$-isomorphism map, i.e., it satisfies Conditions \eqref{eq:preserving}-\eqref{eq:J}. Furthermore, as 
  as already noted in Remark \ref{isomapsremark},  the above properties also  imply that 
 $\Phi(I) = I$, $ \Phi(A_G) = A_H $ and that  $|V_G|=|V_H|=:n$. 
 Additionally, by Theorem~\ref{thm:isomap} the adjoint map $\Phi^\dagger$ is a $\K$-isomorphism map from $H$ to $G$, i.e., it satisfies:
\be\label{daggereq}
 \Phi^\dagger(I \schur X) = I \schur \Phi^\dagger(X)\   \text{ and } \ 
\Phi^\dagger(A_H \schur X) = A_G \schur \Phi^\dagger(X), \text{ for all } X \in \C^{n\times n}. 
\ee
%
%
  Now, \eqref{daggereq} combined with $\Phi^\dagger(J)=J$ imply that 
  $ \Phi^\dag(I) = I $ and  $\Phi^\dag(A_H) = A_G.$
 
  Summarizing, we have determined that  $\Phi_M$ is completely positive, unital, trace-preserving, and    
$$  \Phi(I)=I=\Phi^\dag(I), \quad  \Phi(J)=J=\Phi^\dag(J), \quad 
    \Phi(A_G)=A_H, \quad \Phi^\dag(A_H)=A_G.$$
Consequently, Lemma~\ref{lem:XYcommute} implies that for any $ W\in \C^{n\times n}$ we have that  $\Phi(XW) = \Phi(X)\Phi(W) $  and $\Phi(WX)=\Phi(W)\Phi(X)$ whenever $X\in \{I,J,A_G\}$. 
 Furthermore, by Lemma~\ref{newlemma} and Condition~\eqref{zeros} of isomorphism matrices, we have that for any $W \in \mathbb{C}^{n \times n}$,
 \[\Phi(I \schur W) = I \schur \Phi(W), \ \Phi(A_G \schur W) = A_H \schur \Phi(W), \ \Phi(A_{\overline{G}} \schur W) = A_{\overline{H}} \schur \Phi(W),\]
and similarly for $\Phi^\dag$.
 Consequently,  any finite expression involving $I, A_G, A_{\overline{G}}$ and the operations of addition, scalar multiplication, matrix multiplication, and Schur multiplication where at least one of the factors is $I$ or $A_G$, will be mapped by $\Phi$ to the same expression with $A_G$ and $A_{\overline{G}}$ replaced with $A_H$ and $A_{\overline{H}}$ respectively. Further, $\Phi^\dag$ is the inverse of $\Phi$ on such expressions. This means that the restriction of $\Phi$ to the partially coherent algebra of $G$ is a partial equivalence. 
%

Conversely, suppose that $\phi: \PA_G \to \PA_H$ is a partial equivalence of $G$ and $H$. By Lemma~\ref{lem:unitary}, there exists a unitary matrix $U$ such that $\phi(X) = UXU^\dagger$ for all $X \in \PA_G$. Let $\hphi: \C^{V_G \times V_G} \to \C^{V_H \times V_H}$ be defined as $\hphi(X) = UXU^\dagger$ for all $X \in \C^{V_G \times V_G}$. Obviously, $\phi$ and $\hphi$ agree on $\PA_G$. Also, $\hphi$ is a CPTP unital map with adjoint $\hphi^\dag(X) = \hphi^{-1}(X) = U^\dagger X U$. 
Let $\Pi: \C^{V_G \times V_G} \to \PA_G$ be the orthogonal projection onto the partially coherent algebra of $G$ and define the composition 
\[\Phi = \hphi \comp \Pi: \C^{V_G \times V_G} \to \C^{V_H \times V_H}.\]
Since $\Pi$ is a CPTP unital map by Lemma~\ref{lem:algproj},  we have that $\Phi$ is the composition of two CPTP unital maps and is thus CPTP and unital itself. We  show that $\Phi$ is a $\psd$-isomorphism map for $G$ to $H$, and thus, Theorem \ref{thm:isomap} implies that its Choi matrix is a $\psd$-isomorphism matrix.  

We already have established that $\Phi$ is completely $\psd$-preserving, i.e.,  completely positive.
Since $J \in \PA_G$, we have $\Pi(J) = J$. Also, since $\phi$ is a  partial equivalence it satisfies  $\phi(J)=J$, and consequently  $\hphi(J)=\phi(J)= J$. Therefore, it follows that $\Phi(J)= \hphi \comp \Pi(J)= J$. On the other hand, we have that $\Phi^\dag = \Pi^\dag \comp \hphi^\dag = \Pi \comp \hphi^{-1}$ and thus $\Phi^\dag(J) = J$. So $\Phi$ satisfies property~(\ref{eq:J}). Thus it is only left to show that $\Phi$ satisfies properties~(\ref{eq:verts}) and~(\ref{eq:edges}).

We first aim to show that $\Phi(A_G \schur X) = A_H \schur \Phi(X)$ for all $X \in \C^{V_G \times V_G}$. Let $\Lambda: \C^{V_G \times V_G}\to \C^{V_G \times V_G}$ be the linear map defined by $\Lambda(X) = A_G \schur X$. It is easy to see that $\Lambda$ is a self-adjoint projection onto a subspace of $\C^{V_G \times V_G}$.  Since $A_G \schur X \in \PA_G$ for all $X \in \PA_G$, we have that
\[\Pi \comp \Lambda \comp \Pi = \Lambda \comp \Pi.\]
It follows that
\[\Pi \comp \Lambda = \Pi^\dag \comp \Lambda^\dag = (\Lambda \comp \Pi)^\dag = (\Pi \comp \Lambda \comp \Pi)^\dag = \Pi \comp \Lambda \comp \Pi = \Lambda \comp \Pi,\]
i.e., that $\Lambda$ and $\Pi$ commute. Therefore,
\[\Phi(A_G \schur X) = \hphi \circ \Pi \circ \Lambda(X) =  \hphi \circ \Lambda \circ \Pi(X) =  \hphi(A_G \schur \Pi(X)) = A_H \schur \hphi(\Pi(X)) = A_H \schur \Phi(X),\]
where the second to last equality uses the fact that $\Pi(X) \in \PA_G$. So $\Phi$ satisfies property~(\ref{eq:edges}).

We can similarly show that $\Phi(I \schur X) = I \schur \Phi(X)$ and $\Phi(A_{\overline{G}} \schur X) = A_{\overline{H}} \schur \Phi(X)$, i.e., that $\Phi$ satisfies property~(\ref{eq:verts}). Therefore $\Phi$ is an $\psd$-isomorphism map for $G$ to $H$ and we are done.
%
\qeds

%
%
%
%

\subsection{Necessary conditions for $\psd$-isomorphism}\label{subsec:psdisonecessary}

\begin{lemma}\label{lem:psd2cospec}
If $G\cong_\psd H$ they have cospectral adjacency  matrices, as do their complements.
\end{lemma}

\begin{proof}Assume that $G\cong_\psd H$. By Theorem \ref{thm:psdiso} there 
exists  an $\psd$-isomorphism map  $\Phi$ from  $G$ to $H$. As we have  already seen, the map $\Phi$ satisfies 
$ \Phi(A_G) = A_H, \ \Phi^\dag(A_H)=A_G, \ \Phi(A_{\overline{G}})=A_{\overline{H}}, \ \Phi^\dag(A_{\overline{H}})=A_{\overline{G}},$  and the claim  follows immediately by Lemma~\ref{lem:cospectral}.
\end{proof}
We note that the above result in the special case of 
  quantum isomorphic graphs was proved in~\cite{qiso}. 
Moreover, there are other types of cospectrality that one can consider. Another common  cospectrality relation is in terms of the   (combinatorial) Laplacian of a graph $G$, defined as the matrix $L = D - A_G$ where $D$ is a diagonal matrix of degrees 
 and $A_G$ is the adjacency matrix. 


\begin{lemma}\label{lem:laplacian}
If $G\cong_\psd H$  they have cospectral Laplacian matrices, as do their complements.
\end{lemma}
\proof
It is easy to see that if $A_G$ is the adjacency matrix of $G$, then $I \schur A_G^2 - A_G$ is the Laplacian of $G$. Suppose that $\Phi$ is a $\psd$-isomorphism map for $G$ to $H$. Then we have that
\[\Phi(I \schur A_G^2 - A_G) = I \schur \Phi(A_G)^2 - \Phi(A_G) = I \schur A_H^2 - A_H,\]
which is of course the Laplacian of $H$. Similarly, we have that $\Phi^\dag(I \schur A_H^2 - A_H) = I \schur A_G^2 - A_G$, and by Lemma~\ref{lem:cospectral} this implies that the Laplacians of $G$ and $H$ have the same eigenvalues.\qeds

One can similarly show that $\psd$-isomorphic graphs are cospectral with respect to their signless or normalized Laplacians, as well as many other similarly constructed matrices. An important property of the Laplacian of a graph $G$ is that the number of connected components of $G$ is equal to the multiplicity of zero as an eigenvalue of its Laplacian \cite[Proposition 1.3.7]{brouwer2011spectra}. Therefore, we have the following.

\begin{corollary}\label{lem:conncomps}
If $G\cong_\psd H$  they have the same number of connected components, as do their complements. 
\end{corollary}


Another property preserved by $\psd$-isomorphism has to do with the number of walks in a graph. We say that a graph $G$ is \emph{walk-regular} if the number of walks of length $\ell$ beginning and ending at a vertex of $G$ is independent of the choice of vertex. Equivalently, there exists a number $a_\ell\in \mathbb{N}$ such that $I \schur A_G^\ell = a_\ell I$ for all $\ell \in \mathbb{N}$. We also say that a graph is \emph{1-walk-regular} if it is walk-regular and there exists $b_\ell\in \mathbb{N}$ such that $A_G \schur A_G^\ell = b_\ell A_G$ for all $\ell \in \mathbb{N}$. Obviously, this means that the number of walks of length $\ell$ starting at one end of an edge and ending at the other does not depend on the edge. It turns out that $\psd$-isomorphism preserves both of the aforementioned   properties:

\begin{lemma}\label{lem:1walkregpreserved}
If $G$ and $H$ are $\psd$-isomorphic graphs, then $G$ is walk-regular if and only if $H$ is walk-regular. The same holds for 1-walk-regularity.
\end{lemma}
\proof
Suppose $G$ is walk-regular and let $a_\ell$ for $\ell \in \mathbb{N}$ satisfying $I \schur A_G^\ell = a_\ell I$. By Theorem \ref{thm:psdiso} there  exists   a $\psd$-isomorphism map $\Phi$  from  $G$ to $H$.
Then, we have that
\[I \schur A_H^\ell = \Phi(I \schur A_G^\ell) = \Phi(a_\ell I) = a_\ell I,\]
and thus $H$ is walk-regular. Essentially the same proof works for 1-walk-regularity.\qeds

Walk-regularity also turns out to be related to the partially coherent algebra of a graph. Below we refer to the algebra of polynomials in the adjacency matrix of a graph $G$ as the \emph{adjacency algebra of $G$}.

\begin{lemma}\label{vdfbrgn}
 If the adjacency algebra of a graph $G$ is equal to its partially coherent algebra, then $G$ is connected and walk-regular. The converse implication does not hold. 
 \end{lemma}
\proof
Let $\PA_G$ be the partially coherent algebra of $G$ and assume that this is equal to the adjacency algebra of $G$. Since $J \in \PA_G$ by definition, we have that $J$ is contained in the adjacency algebra of $G$ which happens if and only if $G$ is connected and regular~\cite[Theorem 1]{hoffman1963polynomial}. So it only remains to show that $G$ is walk-regular. 

Consider the subspace $\mathcal{D} = \{I \schur X : X \in \PA_G\}$. By Lemma~\ref{lem:01basis}, there exists an orthogonal basis of $\mathcal{D}$ consisting of diagonal $01$-matrices. Let $\{D_1, \ldots, D_r\}$ be this basis and suppose that $r > 1$. This implies that $D_1$ is not the identity matrix and therefore $D_1J \in \PA_G$ is not symmetric. This contradicts the assumption that $\PA_G$ is equal to the adjacency algebra of $G$, which obviously contains only symmetric matrices. Therefore, we have that $r = 1$ and $\mathcal{D}$ is just the span of the identity matrix. However, since $A_G^\ell \in \PA_G$, we have that $I \schur A_G^\ell \in \mathcal{D}$. Therefore, for any $\ell \in \mathbb{N}$, we have that there exists a number $a_\ell$ such that $I \schur A_G^\ell = a_\ell I$, i.e., $G$ is walk-regular.

To show that the converse does not hold, consider the 10-cycle $C_{10}$, and let $G$ be the graph with vertex set $V(C_{10})$ such that two vertices are adjacent if they are at distance one or two in $C_{10}$ (see Figure~\ref{fig:C10}). Note that $G$ is vertex transitive and therefore walk-regular, and it is obviously connected. We will show that the adjacency matrix of $C_{10}$ is contained in the partially coherent algebra of $G$, but not its adjacency algebra. For the former claim, it is straightforward to show that (or simply compute) the adjacency matrix of $C_{10}$ is equal to $A_G \schur A_G^2 - A_G \in \PA_G$. For the latter, if the adjacency algebra of $G$ contains the adjacency matrix of $C_{10}$, then it contains its entire adjacency algebra. However, the dimension of the adjacency algebra of a graph is equal to the number of distinct eigenvalues of its adjacency matrix. For $C_{10}$ this dimension is 6, but for $G$ it is 5 (by direct computation). Thus the adjacency algebra of the latter cannot contain the adjacency matrix of the former, and we are done.\qeds
\begin{figure}[h]
\begin{center}
\includegraphics[scale=.5]{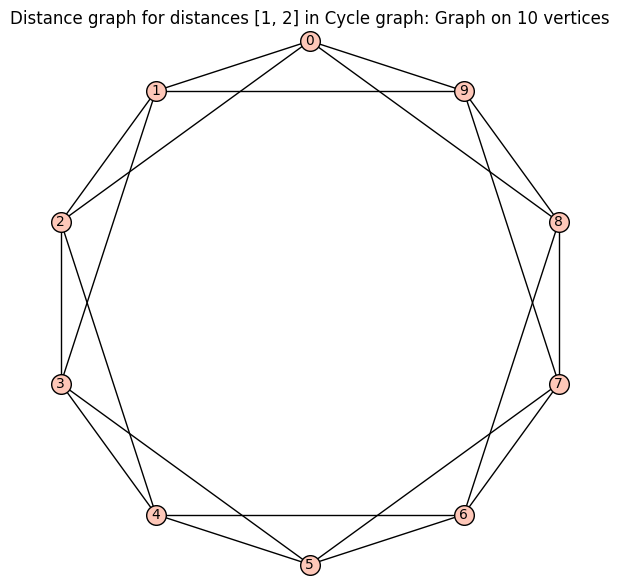}
\caption{Distance 1 and 2 graph of $C_{10}$.}\label{fig:C10}
\end{center}
\end{figure}

If we change walk-regular to 1-walk-regular, then the necessary condition of Lemma \ref{vdfbrgn}  becomes a sufficient condition:

\begin{lemma}\label{lem:1walkregPA}
If $G$ is a connected 1-walk-regular graph, then the partially coherent algebra of $G$ is equal to the adjacency algebra of $G$. The converse does not hold.
\end{lemma}
\begin{proof}

It is obvious that the partially coherent algebra of $G$ contains the adjacency algebra of $G$. To prove the first claim it therefore suffices to show that the adjacency algebra is $S$-partially coherent for $S = \{I, A_G\}$. First, since $G$ is 1-walk-regular, it is regular and moreover  it is also connected  by assumption. Using the known fact that $J$ is contained in the adjacency algebra if and only  if $G$ is connected and  regular~\cite{hoffman1963polynomial}, we have that the all ones matrix $J$ can be written as  a polynomial in $A_G$. Second, it is obvious that the adjacency algebra is closed under conjugate transpose. So it only remains to show that the adjacency algebra is closed under entrywise product with $I$ and $A_G$. Using  the definition of 1-walk-regularity  it follows that for any polynomial $f(x)=\sum_\ell c_\ell x^\ell$, we have that $I \schur f(A_G) = \sum_\ell c_\ell a_\ell I$ and $A_G \schur f(A_G) = \sum_\ell c_\ell b_\ell A_G$, and thus we have proven the claim.

To show that the converse does not hold, consider the 8-cycle $C_8$ and let $G$ be the graph with vertex set $V(C_8)$ such that two vertices are adjacent if they are at distance two or three in $C_8$. We will show that the \emph{coherent algebra} algebra of $G$ is equal to its adjacency algebra, but that it is not 1-walk-regular. Let $\A$ be the coherent algebra of $G$ and let $\mathcal{C}$ be the coherent algebra of $C_8$.

The graph $C_8$ is distance regular and thus both its adjacency and coherent algebras are equal to the span of its distance matrices which obviously contains $A_G$. Thus by minimality of $\A$, we have that $\A \subseteq \mathcal{C}$. On the other hand, $\mathcal{C}$ has dimension 5 since $C_8$ has diameter 4 and the adjacency algebra of $G$, which is contained in $\A$, has dimension 5 since $A_G$ has 5 distinct eigenvalues. Thus we have that the adjacency algebra of $G$ is equal to $\A = \mathcal{C}$.

However, the number of walks of length two between adjacent vertices of $G$ is not constant, it depends on the distance between the vertices in $C_8$. Therefore $G$ is not 1-walk-regular.
\end{proof}

The above two lemmas show that the property of having adjacency algebra equal to partially coherent algebra lies somewhere (strictly) in between being walk-regular and being 1-walk-regular.

\begin{theorem}\label{thm:1walkregpsd}
Let $G$ be a connected 1-walk-regular graph. For any graph $H$ we have that  $G \cong_\psd H$ if and only if $H$ is a connected 1-walk-regular graph that is cospectral to $G$.
\end{theorem}
\proof
If $G \cong_\psd  H$, it follows  that  $H$ is  a connected  (Lemma~\ref{lem:conncomps}) 1-walk-regular graph (Lemma \ref{lem:1walkregpreserved})  that is also 
cospectral to $G$ (Lemma \ref{lem:psd2cospec}). 

Conversely, suppose that $H$ is a connected 1-walk-regular graph that is cospectral to $G$. Since they are cospectral, by the spectral theorem there exists a unitary matrix $U$ such that $UA_GU^\dagger = A_H$. It is then easy to see that the map $\phi(X) = UXU^\dagger$ is an algebra isomorphism from the adjacency algebra of $G$ to that of $H$. By Lemma~\ref{lem:1walkregPA}, it follows  that $\phi$ is an algebra isomorphism from $\PA_G$ to $\PA_H$, it remains  to verify that this is a partial equivalence. Obviously, $\phi(X^\dagger) = \phi(X)^\dagger$, and so this condition is met. We also need that $\phi(J) = J$, but this holds because if $E_\lambda$ and $F_\lambda$ are the projections onto the $\lambda$-eigenspaces of $G$ and $H$ respectively, then $UE_\lambda U^\dagger = F_\lambda$, and $\frac{1}{n}J$ (where $n = |V_G| = |V_H|$) is the projection onto the maximum eigenspaces of both $G$ and $H$ since they are connected and regular. 

Lastly, we show  that $\phi(I \schur X) = I \schur \phi(X)$ and $\phi(A_G \schur X) = A_H \schur \phi(X)$ for all $X \in \PA_G$. As $G$ is a connected 1-walk-regular graph, by Lemma \ref{lem:1walkregPA} the partially coherent algebra of $G$ is equal to the adjacency algebra of $G$ and thus $I \schur X = (\tr(X)/n) I$ for all $X \in \PA_G$. 
Therefore,
\be\label{sDvaeb}
\phi(I \schur X) = \phi\left(\frac{\tr(X)}{n}I\right) = \frac{\tr(X)}{n}I.
\ee
On the other hand, $\phi(X) \in \PA_H$ and thus
\be\label{csddfbd}
I \schur \phi(X) = \left(\frac{\tr(\phi(X))}{n}\right)I = \left(\frac{\tr(X)}{n}\right), 
\ee
where the second equality follows from the fact that $\phi$ is trace-preserving. Thus we have shown that $\phi(I \schur X) = I \schur \phi(X)$ for all $X \in \PA_G$.

We similarly have that $A_G \schur X = \gamma A_G$ where $\gamma = \tr(A_GX)/nk$, where $k$ is the degree of both $G$ and $H$. Thus $\phi(A_G \schur X) = \gamma A_H$. Of course we also have that $A_H \schur \phi(X) = \gamma'A_H$ where
\[\gamma' = \tr(A_H \phi(X))/nk = \tr(\phi(A_G X))/nk = \tr(A_G X)/nk = \gamma.\]
Thus $G$ and $H$ are partially equivalent and by Theorem~\ref{thm:characterization} we have $G \cong_\psd H$.\qeds

\section{Characterizing $\dnn$-isomorphic graphs}\label{sec:characterization}

\subsection{Coherent algebras of graphs}\label{subsec:cohalggraphs}




The \emph{coherent algebra of a graph $G$}, denoted $\A_G$, is defined to be the intersection of all coherent algebras containing its adjacency matrix $A$, {\it i.e.}, the smallest coherent algebra containing $A$. Equivalently, this is all the matrices that can be written as a finite expression involving $I$, $A$, $J$, and the operations of addition, scalar multiplication, matrix multiplication, Schur multiplication, and conjugate transpose. 

An \emph{isomorphism} between coherent algebras $\A$ and $\B$ is a bijective linear map $\phi: \A \to \B$ that preserves all operations of a coherent algebra, {\it i.e.},
\begin{itemize}
\item $\phi(M^\dagger) = \phi(M)^\dagger$ for all $M \in \A$;
\item $\phi(MN) = \phi(M)\phi(N)$ for all $M,N \in \A$;
\item $\phi(M\schur N) = \phi(M)\schur\phi(N)$ for all $M,N \in \A$.
\end{itemize}

As a consequence of the above, we must have that $\phi(I) = I$ and $\phi(J) = J$. More generally, if $\phi$ is an isomorphism of coherent algebras $\A$ and $\B$, and $A_1, \ldots, A_d$ and $B_1, \ldots, B_d$ are the orthogonal 01-matrices forming bases of $\A$ and $\B$ respectively, then there exists a bijection $f: [d] \to [d]$ such that $\phi(A_i) = B_{f(i)}$ for all $i \in [d]$. 

If $G$ and $H$ are two graphs with respective adjacency matrices $A_G$ and $A_H$ and coherent algebras $\A_G$ and $\A_H$, then we say that $G$ and $H$ are \emph{equivalent} if there exists an isomorphism $\phi$ from $\A_G$ to $\A_H$ such that $\phi(A_G) = A_H$. We refer to the map $\phi$ as an \emph{equivalence} of $G$ and $H$. It is known~\cite{WL} that two graphs are equivalent if and only if they are not distinguished by the Weisfeiler-Leman method.

\subsection{The characterization} 

We will need the following:

\begin{lemma}\label{doublystochasticthm} 
Consider a doubly stochastic  matrix   $D \in \mathbb{R}^{d \times d}$ and column vectors   $u, v\in \R^d$ with the same multiset of entries. 
The following are equivalent:
\bi 
\item[(1)]  $Du = v$. 
\item[(2)] $D_{ij} = 0$ whenever $v_i\ne u_j$.
\item[(3)]  $D(u \schur w) = v \schur (Dw)$   for  all vectors $w$.
\item[(4)]  $D^T (v \schur w) = u \schur (D^T w)$ for  all vectors $w$.
\item[(5)] $D^Tv=u$. 
\ei

\end{lemma}
\begin{proof}$(1) \implies (2)$.  Suppose that $Du = v$. 
Set  $V=\{ i\in [n]: v_i = v^{\dar}_1\}$, i.e., the indices of the largest entry of $v$ 
and define $U$ similarly. As $D$ is doubly stochastic, the equation $v_i = (Du)_i = \sum_j D_{ij}u_j$ shows that $v_i$ is a convex combination of the entries of $u$. If $i \in V$,  no entry $u_j$ for $j \notin U$ can appear with nonzero weight in this convex combination. Therefore, we have that 
 \be\label{eevaerve}
 i \in V, \ j \notin U \implies D_{ij} = 0,
 \ee
  and thus 
\be\label{xsdvefbgr}  
1 = \sum_{j \in [n]} D_{ij} = \sum_{j \in U} D_{ij}, \text{ for all } i\in V.
\ee
Furthermore, we have that 
\be\label{fvbnrbtn}
|U| = |V| = \sum_{i \in V} \sum_{j \in U} D_{ij} = \sum_{j \in U} \sum_{i \in V} D_{ij} \le \sum_{j \in U} \sum_{i \in [n]} D_{ij} = |U|,
\ee
where for the first equality we use that $u$ and $v$ have the same multiset of entries and  for the second equality  we use \eqref{xsdvefbgr}. 
Thus, \eqref{fvbnrbtn} holds throughout with equality, which in turn implies  
 that 
 \be\label{dscefbrg}
 j \in U, \ i \notin V \implies D_{ij} = 0.
 \ee 
  Rearranging  $D$ so that the $V$-rows and $U$-columns are first, it follows by \eqref{eevaerve} and \eqref{dscefbrg} that 
\[D = \begin{pmatrix} D' & 0 \\ 0 & D''\end{pmatrix},\]
where $D'$ and $D''$ are doubly stochastic matrices. The same argument  can be applied to  $D''$ (where $V$ and $U$ are now defined as the indices of the second largest entry in $v$ and $u$ respectively). Continuing in the same manner it follows that  $D_{ij} = 0$ whenever $v_i \ne u_j$. 

$(2) \iff (3) $. We have already established this in the proof of Lemma~\ref{newlemma2}.

$(3)\implies (1).$ This follows  by selecting $w=e$. 

Lastly, to get $(4)$ and $(5)$ simply note that $ (2)$ is equivalent to $D^T_{ji}=0$ whenever $u_j\ne v_i$. 
\end{proof} 

As with Lemma~\ref{newlemma2}, we now state a form of the above lemma in terms of maps between matrix spaces. As before this is equivalent to the above by the correspondence $\vect(\Phi(X)) = \tilde{M}\vect(X)$ for where $\Phi$ has Choi matrix $M$ and $\tilde{M}_{hh',gg'} = M_{gh,g'h'}$. Note that $\tilde{M}$ having row and column sums equal to 1 is equivalent to $\Phi(J) = J = \Phi^\dag(J)$.

\begin{lemma}\label{doublystochasticthm2}
Let $\Phi : \mathbb{C}^{V_G \times V_G} \to \mathbb{C}^{V_H  \times V_H}$ be a linear map with entrywise nonneggative Choi matrix $M$ such that $\Phi(J) = J = \Phi^\dag(J)$. For any fixed pair of matrices  $X \in \mathbb{C}^{V_G \times V_G}$ and $Y \in \mathbb{C}^{V_H \times V_H}$ with the same multiset of entries the following are equivalent:
\begin{enumerate}
\item[$(1)$] $\Phi(X) = Y$.
\item[$(2)$] $M_{gh,g'h'} = 0$ whenever $X_{gg'} \ne Y_{hh'}$.
\item[$(3)$] $\Phi(X \schur W) = Y \schur \Phi(W)$ for all $W \in \mathbb{C}^{V_G \times V_G}$.
\item[$(4)$] $\Phi^\dag(Y \schur Z) = X \schur \Phi^\dag(Z)$ for all $Z \in \mathbb{C}^{V_H \times V_H}$.
\item[$(5)$] $\Phi^\dag(Y) = X$.
\end{enumerate}
\end{lemma}

\begin{theorem}\label{thm:characterization}
Two graphs $G$ and $H$ are equivalent if and only if $G \cong_\dnn H$. 
\end{theorem}
\begin{proof}

Suppose that $G \cong_\dnn H$ and  let $M$ be a $\dnn$-isomorphism matrix.  Consider the linear map $\Phi: \C^{V_G \times V_G} \to \C^{V_H \times V_H}$ whose Choi matrix is equal to $M$.
As $M$  is a $\dnn$-isomorphism matrix it follows by Theorem~\ref{thm:isomap} that $\Phi$ is  a $\dnn$-isomorphism map. We show that $\Phi$ is the desired equivalence between the coherent algebras of $G$ and $H$. 
 
 The same arguments as in Theorem \ref{thm:psdiso} apply to show that   any finite expression involving $I, A_G, A_{\overline{G}}$ and the operations of addition, scalar multiplication, matrix multiplication, and Schur multiplication where at least one of the factors is $I$ or $A_G$, will be mapped by $\Phi$ to the same expression with $A_G$ and $A_{\overline{G}}$ replaced with $A_H$ and $A_{\overline{H}}$ respectively. Furthermore, $\Phi^\dag$ is the inverse of $\Phi$ on such expressions. 
It remains to consider the case of  arbitrary Schur products. For this, it suffices to   show that for any pair of 
   matrices $X,Y$ such that $\Phi(X)=Y$ and $\Phi^\dag(Y)=X$ we  have that 
\be\label{phischur}
\Phi(X \schur W) = \Phi(X) \schur \Phi(W) \ \text{ and } \ \Phi^\dag(Y  \schur W)=\Phi^\dag(Y) \schur \Phi^\dag(W), \text{ for all } W. 
\ee
However, since $\Phi$ is a $\dnn$-isomorphism map, we have that its Choi matrix is entrywise nonneggative and $\Phi(J) = J = \Phi^\dag(J)$. This is the matrix map analog of being doubly stochastic, and thus $\Phi(X) = Y$ and $\Phi^\dag(Y) = X$ implies that $X$ and $Y$ have the same multiset of entries. Therefore we can apply Lemma~\ref{doublystochasticthm2} to obtain Equation~\eqref{phischur}.
We thus have that for $X_1,X_2,Y_1,Y_2$ such that $\Phi(X_i) = Y_i$ and $\Phi^\dag(Y_i) = X_i$, it holds that $\Phi(X_1 \schur X_2) = Y_1 \schur Y_2$ and $\Phi^\dag(Y_1 \schur Y_2) = X_1 \schur X_2$. Since $\Phi(I) = I$, $\Phi(A_G) = A_H$, $\Phi(J) = J$, and similarly for $\Phi^\dag$, it follows that any expression in $I, A_G, J$ using addition, scalar multiplication, matrix multiplication, conjugate transposition, and Schur product is mapped by $\Phi$ to the same expression but with $A_G$ replaced with $A_H$ (and $\Phi^\dag$ is the inverse of $\Phi$ on such expressions). Therefore the restriction of $\Phi$ to $\A_G$ is an equivalence of $G$ and $H$.


Conversely, let  $\phi: \A_G \to \A_H$ be  an  equivalence of $G$ and $H$. By Lemma~\ref{lem:unitary}, there exists a unitary matrix $U$ such that $\phi(X) = UXU^\dagger$ for all $X \in \A_G$. Let $\hphi: \C^{V_G \times V_G} \to \C^{V_H \times V_H}$ be defined as $\hphi(X) = UXU^\dagger$ for all $X \in \C^{V_G \times V_G}$. Moreover, let  $\Pi: \C^{V_G \times V_G} \to \A_G$ be the orthogonal projection onto the  coherent algebra of $\A_G$. The same arguments as in the proof of  Theorem \ref{thm:psdiso} imply  that the composition  
$\Phi = \hphi \comp \Pi: \C^{V_G \times V_G} \to \C^{V_H \times V_H}$
is an $\psd$-isomorphism map for $G$ to $H$. Thus, to show that $\Phi$ is a $\dnn$-isomorphism 
 map it   remains to show that it is completely $\dnn$-preserving, or equivalently, that its  Choi matrix  $C_\Phi$ is entrywise nonnegative (we already know it is psd). 
 
  

Let $A_1, \ldots, A_d$ and $B_1, \ldots, B_d$ be the 01-bases of $\A_G$ and $\A_H$ respectively. Then, we have that  $\phi(A_i) = B_i$ and  as $\phi$ is trace preserving (cf. Lemma \ref{pequiv_trace_pres}) it follows that $m_i = \tr(A_i^TA_i) = \tr(B_i^TB_i)$, where $m_i$ is the number of 1's in $A_i$ and $B_i$. Then,   $\{\frac{1}{\sqrt{m_i}} A_i : i \in [d]\}$ is an orthonormal basis for $\A_G$, and thus

\[\Pi(X) = \sum_{i=1}^d \frac{1}{m_i}\langle A_i, X\rangle A_i = \sum_{i=1}^d \frac{1}{m_i}\tr(A_i^T X) A_i,\]
which in turn implies that
\be\label{csdcdfg}
\Phi(X) = \sum_{i=1}^d \frac{1}{m_i}\tr(A_i^T X) B_i.
\ee
By \eqref{csdcdfg} we clearly see that   if $X$ is entrywise nonnegative, then $\Phi(X)$ is also nonnegative.  Thus the Choi matrix of $\Phi$ is doubly nonnegative and so $\Phi$ is a $\dnn$-isomorphism map for $G$ to $H$.
\end{proof}

\subsection{Necessary conditions on $\dnn$-isomorphic graphs}\label{subsec:dnnconds}

Since any pair of $\dnn$-isomorphic graphs are also $\psd$-isomorphic, we know that any of the necessary conditions for $\psd$-isomorphism given in the previous section are also necessary for $\dnn$-isomorphism. However, some of these necessary conditions can be strengthened.

%


We say that the $d$-distance graph of $G$ is the graph with vertex set $V_G$ such that two vertices are adjacent if their distance in $G$ is exactly $d$. The $d$-distance matrix of $G$ is the adjacency matrix of its $d$-distance graph, so in particular, it has zero diagonal.

%

\begin{lemma}\label{cor:distmats}
Consider two graphs $G$ and $H$. Define   $X^{\ell,i}$ (respectively $Y^{\ell,i}$) as the matrix whose $gg'$-entry is 1 if the number of walks of length $\ell$ in $G$ from $g$ to $g'$ is equal to $i$, and is otherwise zero. 
Moreover, let  $X^{(\ell)}$ and $Y^{(\ell)}$ be the $\ell$-distance matrices of $G$ and $H$ respectively.
Assume  that $G$ and $H$ are $\dnn$-isomorphic graphs with isomorphism map $\Phi$.  Then  we have that 
$$ 
 \Phi(X^{\ell,i}) = Y^{\ell,i} \text{ for all } \ell, i \in \mathbb{N}\  \text{ and } \ \Phi(X^{(\ell)}) = Y^{(\ell)} \text{ for all } \ell = 0, 1, \ldots, \mathrm{diam}(G)$$
\end{lemma}
\begin{proof}
We have that $\Phi(A_G^\ell) = A_H^\ell$ and $\Phi^\dag(A_H^\ell) = A_G^\ell$ and thus $A_G^\ell$ and $A_H^\ell$ have the same multiset of entries. Let $S$ be the set of entries of $A_G^\ell$. Then $A_G^\ell = \sum_{i \in S} iX^{\ell,i}$ and $A_H^\ell = \sum_{i \in S} iY^{\ell,i}$. It is then easy to see that for any $j \in S$,
\[X^{\ell,j} = \schur_{i \ne j}\frac{1}{j-i}(A_G^\ell - iJ),\]
and similarly for $Y^{\ell,j}$. It then follows from the properties of $\dnn$-isomorphism maps that $\Phi(X^{\ell,i}) = Y^{\ell,i}$ for all $\ell,i$.

The second claim holds for $\ell = 0,1$, and we proceed by induction. It is easy to see that
\begin{equation}\label{eqn:distmats}
X^{(\ell)} = \left(\sum_{i \ge 1} X^{\ell,i}\right) \schur \left(J - \sum^{\ell-1}_{k = 0} X^{(k)}\right),
\end{equation}
and thus the claim follows from the properties of $\Phi$ and the obvious induction argument.
\end{proof}

%
%

We saw in  Section~\ref{subsec:psdisonecessary} that $\psd$-isomorphism preserves the property of being 1-walk-regular. For $\dnn$-isomorphism, an even stronger property known as distance regularity is preserved. A connected graph $G$ of diameter $d$ is distance regular if there exist integers $p_{ij}^k$ for $i,j,k \in \{0,1,\ldots,d\}$ such that the number of vertices $w$ at distance $i$ from $u$ and distance $j$ from $v$ is equal to $p_{ij}^k$ whenever $u$ and $v$ are at distance $k$, i.e., whenever ${\rm dist}(u,v)=k$ we have that $|N_i(u)\cap N_j(v)|=p^k_{ij}. $  Letting $X^{(\ell)}$ for $\ell \in \{0,1,\ldots,d\}$ be the $\ell$ distance matrix of $G$, one can see that this definition is equivalent to the equations
\[X^{(i)}X^{(j)} = \sum_k p_{ij}^k X^{(k)}.\]

In other words, the distance matrices are a 01 orthogonal basis of the algebra they generate. From here it is easy to see that this is a coherent algebra. Furthermore, since the distance matrices of $G$ are contained in the coherent algebra of $G$ by Equation~(\ref{eqn:distmats}), the coherent algbera generated by the distance matrices of $G$ is equal to the coherent algebra of $G$. Thus a graph is distance regular if and only if its coherent algebra is equal to the span of its distance matrices, in which case all matrices in the coherent algebra are symmetric and thus they all commute. This allows us to prove the following:

\begin{lemma}\label{lem:distregpreserved}
If $G \cong_\dnn H$ then $G$ is distance regular if and only if $H$ is distance regular.
\end{lemma}
\proof
By assumption, there exists a $\dnn$-isomorphism map  $\Phi=\Phi_M$, where $M$ is a $\dnn$-isomorphism matrix. 
It suffices to show that $G$ being distance regular implies that $H$ is distance regular. So let $G$ be distance regular and let $X^{(\ell)}$ and $Y^{(\ell)}$ be the $\ell$-distance matrices of $G$ and $H$ respectively. By Lemma~\ref{cor:distmats}, we have that $\Phi(X^{(\ell)}) = Y^{(\ell)}$ for all $\ell$. Since $G$ is distance regular, the $X^{(\ell)}$ form a basis of $\A_G$. Since the image of $\A_G$ under $\Phi$ is $\A_H$, and since the restriction of $\Phi$ to $\A_G$ is a linear bijection, we have that $\Phi$ maps any basis of $\A_G$ to a basis of $\A_H$. Therefore, the distance matrices of $H$ form a basis of $\A_H$ and thus $H$ is distance regular.\qeds

In Section~\ref{subsec:psdisonecessary} we showed that the partially coherent algebra of a connected 1-walk-regular graph is equal to its adjacency algebra. Analogously, it is well known that the coherent algebra of a distance regular graph is equal to its adjacency algebra~\cite{BCN}. We will not give a proof of this here, but it suffices to show that the distance matrices of a distance regular graph are polynomials in its adjacency matrix, and this can be done with induction. As one might expect, this allows us to prove an analog of Theorem~\ref{thm:1walkregpsd} for $\dnn$-isomorphism.

\begin{theorem}\label{thm:distregdnn}
Let $G$ be a distance regular graph. If $H$ is a graph, then $G \cong_\dnn H$ if and only if $H$ is a distance regular graph that is cospectral to $G$.
\end{theorem}
\proof
The only if direction follows from Lemmas~\ref{lem:psd2cospec} and~\ref{lem:distregpreserved}. For the other direction, suppose that $G$ and $H$ are cospectral distance regular graphs with eigenvalues $\lambda_1 \ge \ldots \ge \lambda_n$ for $n = |V_G| = |V_H|$, and let $d$ be the diameter of both graphs (it is well known that the diameter of a distance regular graph is one less than the number of distinct eigenvalues~\cite{BCN}, so this will be the same for $G$ and $H$). Note that the unique eigenvector for $\lambda_1$ is the constant vector since the graphs are connected and regular. Also let $X^{(\ell)}$ and $Y^{(\ell)}$ be the $\ell$-distance matrices of $G$ and $H$ respectively. Since $G$ and $H$  are cospectral, there exists a unitary matrix $U$ such that $UA_GU^\dagger = A_H$. It is then immediate that the map $\phi(X) = UXU^\dagger$ is an algebra isomorphism of the adjacency algebras of $G$ and $H$, which are respectively equal to their coherent algebras since the graphs are distance regular. We aim to show that $\phi$ is an equivalence. 

Obviously, $\phi(X^\dagger) = \phi(X)^\dagger$, and $\phi(J) = J$ since $\frac{1}{n}J$ is the projection onto the $\lambda_1$-eigenspace for both $G$ and $H$. So we only need to show that $\phi(X \schur X') = \phi(X) \schur \phi(X')$ for all $X,X' \in \A_G$. 

For any $X,X' \in \A_G$, we have that $X = \sum_\ell \alpha_\ell X^{(\ell)}$ and $X' = \sum_\ell \alpha'_\ell X^{(\ell)}$ for some coefficients $\alpha_\ell, \alpha'_\ell$ for $\ell = 0, \ldots d$ since the distance matrices of $G$ span $\A_G$ by the distance regularity of $G$. Then $X \schur X' = \sum_{\ell} \alpha_\ell \alpha'_\ell X^{(\ell)}$, since $X^{(\ell)} \schur X^{(k)} = \delta_{\ell k} X^{(\ell)}$. Suppose that $\phi(X^{(\ell)}) = Y^{(\ell)}$. Then,
\begin{align*}
\phi(X \schur X') &= \phi\left(\sum_{\ell = 0}^d \alpha_\ell \alpha'_\ell X^{(\ell)} \right) \\
&= \sum_{\ell = 0}^d \alpha_\ell \alpha'_\ell \phi\left(X^{(\ell)} \right) \\
&= \sum_{\ell} \alpha_\ell\alpha'_\ell Y^{(\ell)} \\
&= \left(\sum_{\ell=0}^d \alpha_\ell Y^{(\ell)}\right)  \schur \left(\sum_{\ell=0}^d \alpha'_\ell Y^{(\ell)}\right) \\
&= \phi\left(\sum_{\ell=0}^d \alpha_\ell X^{(\ell)}\right) \schur \phi\left(\sum_{\ell=0}^d \alpha'_\ell X^{(\ell)}\right) \\
&= \phi(X) \schur \phi(X').
\end{align*}
Thus it suffices to show that $\phi(X^{(\ell)}) = Y^{(\ell)}$ for all $\ell = 0, \ldots, d$. Since $G$ is distance regular, there exist polynomials $f_\ell$ for $\ell = 0, \ldots, d$ such that $f_\ell(A_G) = X^{(\ell)}$~\cite[Section 2.7]{BCN}. Moreover, the polynomials $f_\ell$ only depend on the eigenvalues of the distance regular graph. Since $G$ and $H$ are cospectral, we have that $f_\ell(A_H) = Y^{(\ell)}$. Since $\phi(f(A_G)) = f(\phi(A_G)) = f(A_H)$ for any polynomial $f$, it follows that $\phi(X^{(\ell)}) = \phi(f_{\ell}(A_G)) = f_\ell(A_H) = Y^{(\ell)}$ as desired.\qeds

\section{Separations between the various notions of isomorphism}\label{sec:separations}

In Equation~\eqref{eq:chain} we noted that the four  different types of isomorphisms we consider in this paper satisfy the following chain of implications:
\[G \cong H \ \Rightarrow \ G \cong_q H \ \Rightarrow \ G \cong_\dnn H \ \Rightarrow \ G \cong_\psd H .\]
In our earlier work~\cite{qiso}, we showed that the first implication cannot be reversed, i.e., that there are quantum isomorphic graphs that are not isomorphic. Here we show that none of the other implications can be reversed. We begin by showing that $\psd$-isomorphism does not imply $\dnn$-isomorphism.



The first graph we consider is  the 4-cube: the graph whose vertices are the binary strings of length 4, two being adjacent if they differ in exactly one position. This is a well-known distance transitive (and therefore distance regular) graph. Less well-known is the Hoffman graph, which is the unique graph cospectral to the 4-cube. 
 This Hoffman graph is not distance regular but is 1-walk-regular. Both graphs are shown in Figure~\ref{fig:4cubehoff}. 

\begin{figure}\label{fig:4cubehoff}
\begin{center}
\includegraphics[scale=.6]{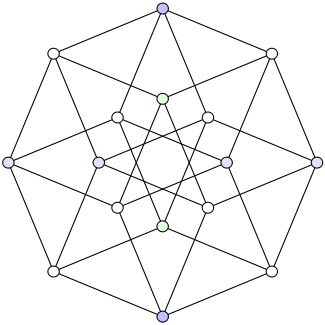}
\hspace{1cm}
\includegraphics[scale=.24]{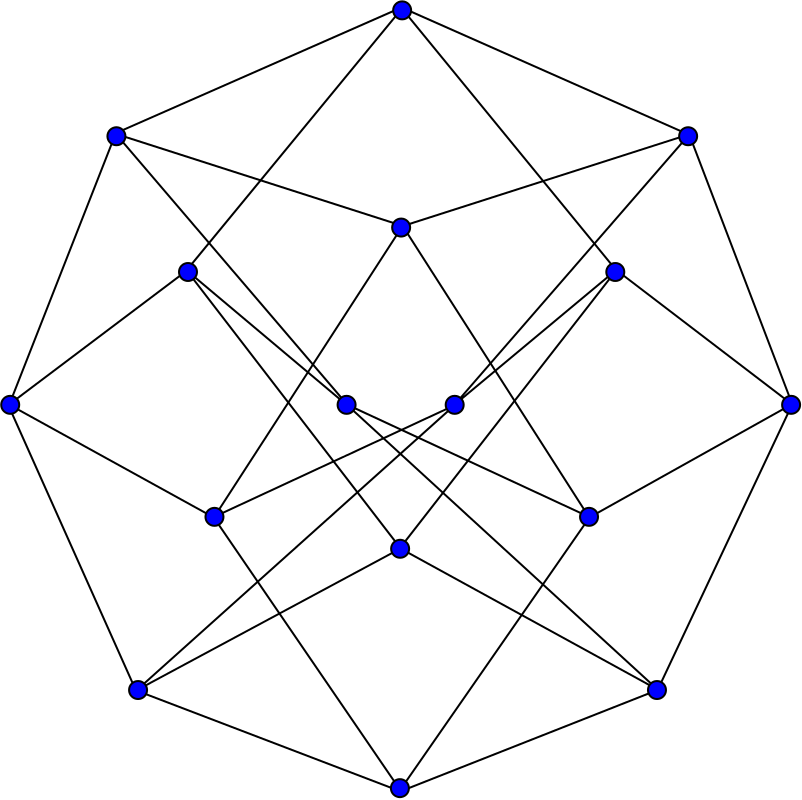}
\caption{The 4-cube and Hoffman graphs.}
\end{center}
\end{figure}

\begin{theorem}\label{thm:psdnotdnn}
There exist graphs $G$ and $H$ that are $\psd$-isomorphic but not $\dnn$-isomorphic, i.e., $G \cong_\psd H \ \not\Rightarrow \ G \cong_\dnn H$.
\end{theorem}
\proof
Let $G$ be the 4-cube and $H$ the Hoffman graph. Since $G$ is distance regular it is 1-walk-regular, and we already noted above that the Hoffman graph is 1-walk-regular. Since they are cospectral, by Theorem~\ref{thm:1walkregpsd} they are $\psd$-isomorphic. However, since $G$ is distance regular but $H$ is not, by Theorem~\ref{thm:distregdnn} they are not $\dnn$-isomorphic.\qeds


The next separation we will show is between quantum and $\dnn$-isomorphism. 
 The graphs we will use are the the cartesian product of $K_4$ with itself, and the Shrikhande graph. The cartesian product of graphs $G$ and $H$, denoted $G \square H$, has vertex set $V_G \times V_H$, and vertices $(g,h)$ and $(g',h')$ are adjacent if ($g = g'$ and $h \sim h'$) or ($g \sim g'$ and $h = h'$). For $G = H = K_4$, the vertices of the graph can be thought of as being in a $4 \times 4$ grid, such that two are adjacent if they are in the same row or column. From this description it is easy to see that $K_4 \square K_4$ contains $K_4$ as a subgraph.

The graph $K_4 \square K_4$ is what is known as a strongly regular graph, which is just a distance regular graph of diameter two. Equivalently, an $n$-vertex, $k$-regular graph $G$ is strongly regular if there exists numbers $\lambda$ and $\mu$ such that any two adjacent vertices of $G$ share $\lambda$ common neighbors, and any two distinct non-adjacent vertices of $G$ share $\mu$ common neighbors. The numbers $(n,k,\lambda, \mu)$ are called the parameters of a strongly regular graph and these completely determine its spectrum. The parameters of $K_4 \square K_4$ are $(16,6,2,2)$.

The Shrikhande graph is pictured in Figure~\ref{fig:Shrikhande}. The Shrikhande graph is also a strongly regular graph with parameters $(16,6,2,2)$, and thus it is cospectral to $K_4 \square K_4$. We will show that these two graphs are $\dnn$-isomorphic but not quantum isomorphic, thus separating these two relations.

\begin{figure}\label{fig:Shrikhande}
\begin{center}
\includegraphics[scale = .7]{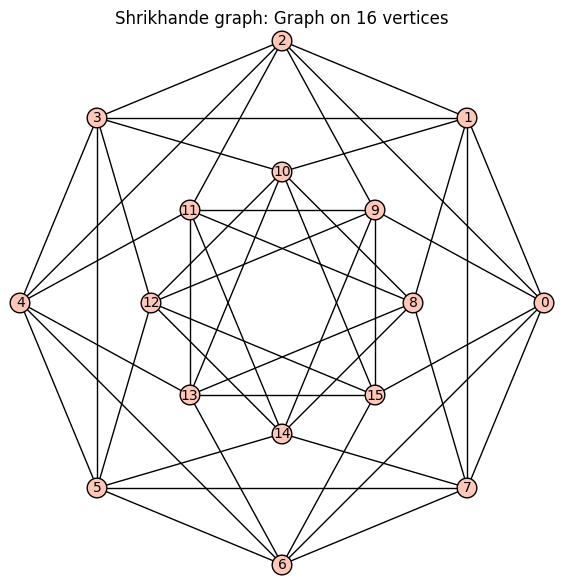}
\caption{Shrikhande graph.}
\end{center}
\end{figure}

\begin{theorem}\label{thm:qdnnsep}
There exist graphs $G$ and $H$ that are $\dnn$-isomorphic but not quantum isomorphic, i.e., $G \cong_\dnn H \ \not\Rightarrow \ G \cong_q H$.
\end{theorem}
\proof
The Shrikhande graph and $K_4 \square K_4$ are strongly regular graphs with the same parameters and are thus cospectral. It then follows from Theorem~\ref{thm:distregdnn} that $G \cong_\dnn H$. It remains to show that these two graphs are not quantum isomorphic.

In~\cite{qplanar}, it is shown that if $G \cong_q H$ then $G$ and $H$ admit the same number of \emph{homomorphisms} from any planar graph $K$\footnote{In fact, it is shown that the latter is equivalent to \emph{quantum commuting isomorphism}, a relaxation of quantum isomorphism arising from considering a different model of joint measurements on shared quantum systems.}. A homomorphism from $K$ to $G$ is simply an adjacency-preserving map from $V_K$ to $V_G$. In the case of $K$ being a complete graph, the existence of a homomorphism from $K$ to $G$ is equivalent to $G$ containing $K$ as a subgraph. It is easy to see that the graph $K_4 \square K_4$ contains $K_4$ as a subgraph, whereas the Shrikhande graph does not. Thus the Shrikhande graph admits no homomorphisms from $K_4$, whereas $K_4 \square K_4$ admits some positive number of homomorphisms from $K_4$. As $K_4$ is planar, it follows that $K_4 \square K_4$ is not quantum isomorphic to the Shrikhande graph.\qeds

We remark that the result of~\cite{qplanar} used to prove the above theorem is a bit of overkill. In fact, the above can be proved more directly using the notion of \emph{projective representations} of graphs and~\cite[Lemma 5.20]{qiso}. However, this more direct proof is somewhat tedious and the above proof allows us to avoid introducing projective representations here.

The above results show that none of the implications in Equation~\eqref{eq:chain} can be reversed. Before moving on, we briefly mention that it is not too difficult to show that $G \cong_\psd H$ implies that $G$ and $H$ are \emph{fractionally isomorphic}, i.e., there exists a doubly stochastic matrix $D$ such that $A_GD = DA_H$. If $M$ is an $\psd$-isomorphism matrix for $G$ to $H$ then letting $D_{gh} = M_{gh,gh}$ works (the proof is similar to~\cite[Lemma 4.2]{qiso}). Again, the reverse implication does not hold, for instance the 6-cycle and the disjoint union of two 3-cycles are fractionally isomorphic, but not $\psd$-isomorphic since they are not cospectral.

\section{Conic theta functions}\label{subsec:conicthetas}


For any matrix cone $\K$, one can define the graph parameter~\cite{dukanovic2010copositive,laurent2015conic}:
\begin{equation*}\label{eqn:thetak}
\vartheta^\K(G) =  \begin{array}[t]{ll}
\sup & \tr(MJ) \\
\text{s.t.} & M_{g,g'} = 0 \text{ if } g \sim g' \\
 & \tr(M) = 1 \\
 & M \in \K. 
\end{array}
\end{equation*}
Note that $\tr(MJ)$ is equal to the sum of the entries of $M$, which we also denote by $\text{sum}(M)$. For $\K = \psd$, the parameter $\thetak$ is exactly  the celebrated  Lov\'{a}sz theta function, denoted simply $\vartheta$. For $\K = \dnn$, it is equal to a variant due to Schrijver that is usually denoted $\vartheta'$ or $\varthetam$.
A nontrivial result~\cite{Pasechnik02} is that $\vartheta^\cp$ is equal to the independence number of a graph, denoted $\alpha$. The parameter $\vartheta^\cpsd$ was first considered in~\cite{laurent2015conic}; it may or may not be equal to a related parameter known as the projective packing number~\cite{RobersonThesis}, but in general it is much less understood than the other three parameters discussed above. One of the main reasons for this is that the cone $\cpsd$ is not closed~\cite{slofstra17}.


Note that if $\K \subseteq \K'$ then $\thetak(G) \le \vartheta^{\K'}(G)$. In particular, we have the following chain of inequalities: 
\[\alpha(G) = \vartheta^\cp(G) \le \vartheta^\cpsd(G) \le \vartheta^\dnn(G) \le \vartheta^\psd(G) = \vartheta(G) \le \chi(\overline{G}).\]

In order to reformulate $\K$-isomorphism in terms of the graph parameter $\thetak$  we  make use of the  \emph{graph isomorphism product}, denoted $G \diamond H$, which has vertex set $V_G \times V_H$ and edges  $(g,h) \sim (g',h')$ if $\rel(g,g') \ne \rel(h,h')$.
In other words, vertices of $G \diamond H$ are adjacent exactly when the corresponding entry in an isomorphism matrix for $G$ to $H$ is required to be zero. Note that the isomorphism product of $G$ and $H$ is the complement of  the \emph{modular} product of graphs.

\begin{theorem}\label{thm:kiso2ktheta}
Consider two graphs  $G$ and $H$ and a matrix cone $\K \subseteq \psd$. Then $G \cong_\K H$ if and only if $\thetak(G \diamond H) = |V_G| = |V_H|$ and this value is attained. 
\end{theorem}
\proof
First, note that $\thetak(G \diamond H) \le \vartheta(G \diamond H) \le \chi(\overline{G \diamond H})$. Also, the set $V_g = \{(g,h) : h \in V_H\}$ is an independent set in $\overline{G \diamond H}$ for all $g \in V_G$. Therefore, $\chi(\overline{G \diamond H}) \le |V_G|$, and similarly $\chi(\overline{G \diamond H}) \le |V_H|$. Thus to show that $\thetak(G \diamond H) = |V_G|$ (or $|V_H|$), it suffices to find a solution of value $|V_G|$ (or $|V_H|$).

Suppose that $G \cong_\K H$ and let be  $M$  the requisite $\K$-isomorphism matrix. We  show that $M$ is, up to a scalar, a feasible solution for $\thetak(G \diamond H)$ of value $|V_G| = |V_H|$. First, by conditions~(\ref{gsums}) and~(\ref{hsums}), we have that $\text{sum}(M)=|V_G|^2=|V_H|^2$, and thus $|V_G| = |V_H|$. Next, we have that $M_{gh,g'h'} = 0$ if $\rel(g,g') \ne \rel(h,h')$ by definition, and thus $M_{gh,g'h'} = 0$ if $(g,h) \sim (g',h')$ in $G \diamond H$. We also have that $M \in \K$. Lastly, it follows again by (\ref{hsums}) that 
\[\tr(M) = \sum_{g \in V_G, \ h \in V_H} M_{gh,gh} = \sum_{g \in V_G} \sum_{h,h' \in V_H} M_{gh,gh'} = |V_G|.\]
Thus, setting 
 $M' = M/\tr(M)$, it follows that  $\tr(M') = 1$ and $\text{sum}(M') = |V_G|$, i.e., $M'$ is a feasible solution for $\thetak(G \diamond H)$ of value $|V_G| = |V_H|$.

Conversely, let  $M$ be an optimal solution 
for $\thetak(G \diamond H)$ of value $|V_G| = |V_H|$.
 We show that the block sums 
\[\sum_{h,h' \in V_H} M_{gh,g'h'} \quad \& \quad \sum_{g,g' \in V_G} M_{gh,g'h'}\]
are all equal to the same constant,
and thus  $M$ is (a scalar multiple of) a $\K$-isomorphism matrix for $G$ to $H$.   
 For  this, define $\widehat{M}$ to be a matrix with rows and columns indexed by $V_G$, whose $g,g'$ entry is given by:
\[\widehat{M}_{g,g'} = \sum_{h,h' \in V_H} M_{gh,g'h'}.\]
First, note that as $M \in \psd$ we also  have that $\widehat{M} \in \psd$. Indeed, if $\{v_{gh}\}_{g \in V_G, h\in V_H}$ is a Gram decomposition of $M$, then $\{\sum_h v_{gh}\}_{g \in V_G}$ is a Gram decomposition for the matrix  $\widehat{M}$. Moreover, by definition  of $\widehat{M}$ we have that  $\text{sum}(\widehat{M}) = \text{sum}(M)$, and using that  $ \text{sum}(M) = |V_G|$
we get that 
\[\frac{e^T \widehat{M} e}{e^T e} = {\text{sum}(\widehat{M}) \over |V_G|}= 1.\]
Consequently,  the maximum eigenvalue of $\widehat{M}$ is at least 1. On the other hand, $$\tr(\widehat{M}) = \sum_{g\in V_G} \sum_{h,h' \in V_H} M_{gh,gh'}=\sum_{g\in V_G} \sum_{h, \in V_H} M_{gh,gh}=\tr(M) = 1,$$ 
where we used that  $M_{gh,gh'} = 0$ when $h \ne h'$. Thus,  the sum of the eigenvalues of $\widehat{M}$ is 1. Since $\widehat{M}$ is positive semidefinite, it must have exactly one nonzero eigenvalue, which must be equal to 1, and which has $e$ as an eigenvector. This implies that $\widehat{M}$ is a multiple of the all ones matrix. By the definition of $\widehat{M}$, this implies that the block sums $\sum_{h,h'} M_{gh,g'h'}$ are constant (and equal to $1/|V_G|$). 
 The same argument shows that the block sums $\sum_{g,g'} M_{gh,g'h'}$ are all equal to $1/|V_H| = 1/|V_G|$, and thus we are done.\qeds

 \section{Connection with the Lasserre hierarchy}\label{sec:las}
Consider a semialgebraic set $\mathcal{K}=\{x\in [0,1]^n: g_i(x)\ge 0, \  i=1,\ldots, m\}$. The Lasserre hierarchy is a systematic method for producing tighter approximations to ${\rm conv}(\mathcal{K}\cap \{0,1\}^n)$. As the goal is to characterize the convex hull of $\{0,1\}$  points in $\mathcal{K}$, using  that $x_i^2=x_i$, we may assume that 
$$g_i(x)=\sum_{K\subseteq [n]} g_i(K)\prod_{i\in K}x_i,$$
i.e., $g_i$ is multilinear and its  monomials are  indexed by subsets of $[n]$.   
For $t\ge 0$, the $t$-th level of the Lasserre hierarchy is  an SDP defined as the set of all vectors 
$y=(y_I), \ I\subseteq [n], |I|\le 2t, $ that satisfy:
$$ M_t(y)\succeq 0,  \quad   M_{t-\lceil {{\rm deg}(g_i)\over 2}\rceil }(g_i*y) \succeq 0  \ (i\in [m]), \quad  y_\emptyset =1.$$
For completeness, we recall that $M_t(y)$ is a matrix indexed  by all sets $I\subseteq [n]$ with $|I|\le t$ and its  $I, J$ entry is given by $y_{I\cup J}$. 
  Furthermore,  $ M_{t-\lceil {{\rm deg}(g_i)\over 2}\rceil }(g_i* y)$ is a matrix indexed by all $I\subseteq [n]$ with $|I|\le t-{{\rm deg}(g_i)\over 2}$ and its  $I,J$ entry is given by 
$\sum_{K\subseteq [n]}g_i(K)y_{I\cup J\cup K}.$

Deciding graph isomorphism is equivalent to the feasibility of the following  quadratic integer program:
 \begin{align}
&\label{new1} \sum_g X_{gh}=1,\\
&\sum_hX_{gh}=1,\\
&  \label{new3} X_{gh}X_{g'h'}=0 \text{ if } \rel(g,g')\ne \rel(h,h'),\\
& X_{gh}\in\{0,1\}, \ \forall g,h.
\end{align}

We proceed to apply the Lasserre hierarchy to the semi-algebraic set obtained  by dropping the integrality constraints. 
We only consider the first level (i.e., $t=1$),  defined in terms of  the  variables  
$$y_\emptyset, \quad y_{(g,h)},  \quad  y_{\{(g,h), (g',h')\}}.$$
The first constraint is that $M_1(y)\succeq 0$. This matrix is  indexed by $\emptyset, (g,h)$; The entries in the 0-th row are $y_{(g,h)}$  and the $(g,h), (g'h')$ entry is $y_{\{(g,h), (g',h')\}}$. Moreover,  
the diagonal is equal to the 0-th row. 
 To handle equality constraints in the Lasserre hierarchy we write them as two inequalities.  As an example consider  $\sum_h X_{gh}- 1\ge 0$. As this has degree one (and we consider level $t=1$), we get the constraint $M_0(g_i*y)\succeq  0$. This is a trivial matrix; is it only indexed by the empty set, so it is just the  scalar 
\be\label{cdfvrgb}
\sum_{K\subseteq [n]}g_i(K)y_{K}.
\ee
Specializing to the polynomial $\sum_h X_{gh}- 1\ge 0$, \eqref{cdfvrgb}  gives the constraint:
$$\sum_h y_{(gg,h)}-y_\emptyset =\sum_h y_{(g,h)}-1\ge 0,$$
whereas the polynomial $\sum_h X_{gh}- 1\le 0$ gives  the converse  inequality. Thus, the first level of the Lasserre hierarchy  includes the constraints
\be\label{xascdf}
\sum_h y_{(g,h)}=1,  \ \forall g,
\ee
and symetrically, it also includes 
$$\sum_g y_{(g,h)}=1, \ \forall h.$$
Finally, the constraints $X_{gh}X_{g'h'}=0 \text{ if } \rel(g,g')\ne \rel(h,h')$ translate to 
$$y_{(g,h),(g',h')} =0 \text{ if } \rel(g,g')\ne \rel(h,h').$$ 

\begin{lemma} $G\cong_\psd H$  if and only if the first level of the Lasserre hierarchy for graph isomorphism is feasible, i.e., there exists $y=(y_\emptyset,  y_{(g,h)},   y_{\{(g,h), (g',h')\}})$ such that 
\begin{align} 
& \label{lasfirst} M_1(y)\succeq 0,\\ 
&  \sum_h y_{(g,h)}=1,  \ \forall g, \\
& \sum_g y_{(g,h)}=1, \ \forall h,\\
& \label{lasrel} y_{(g,h),(g',h')} =0, \text{ if } \rel(g,g')\ne \rel(h,h'),\\
& \label{laslast} y_\emptyset=1.
\end{align}
Furthermore, $\dnn$-isomorphism is equivalent to the feasibility of \eqref{lasfirst}-\eqref{laslast} with the additional constraints that the variables are nonnegative.  
\end{lemma}
\begin{proof}Let $y=(y_\emptyset,  y_{(g,h)},   y_{\{(g,h), (g',h')\}})$ feasible for \eqref{lasfirst}-\eqref{laslast}. As $M_1(y)$ is positive semidefinite, it can be realized as the Gram  matrix of  a family of vectors $v_\emptyset $ and $v_{(g,h)}, g\in V_G, h\in V_H$. The main step is to show  that $v_\emptyset=\sum_g v_{(g,h)}=\sum_h v_{(g,h)}$. Towards  this end,   note that 
$$\la \sum_h v_{(g,h)}, \sum_h v_{(g,h)}\ra =\sum_h \la v_{(g,h)}, v_{(g,h)}\ra =\sum_h y_{(g,h)}=1,$$
where the first equality follows from \eqref{lasrel}, and 
$$\la v_\emptyset,\sum_h  v_{(g,h)}\ra =\sum_h y_{(g,h)}=1.$$
Combining the above  with $y_\emptyset=1$ we get that 
$$\|v_\emptyset -\sum_h  v_{(g,h)}\|^2=0, \ \forall g,$$
and analogously  that  $\|v_\emptyset -\sum_g v_{(g,h)}\|^2=0, \ \forall h.$ Lastly, it follows that the restriction of $M_1(y)$ on the rows/columns indexed by $(g,h)$ is a $\psd$-isomorphism matrix as  
$$\sum_{h,h' \in V_H} M_1(y)_{gh,g'h'} =\sum_{h,h'\in V_H}\la v_{(g,h)}, v_{(g',h')}\ra = \la v_\emptyset, v_\emptyset\ra =1, \text{ for all } g,g' \in V_G,$$
and similarly $\sum_{g,g'} M_1(y)_{gh,g'h'} = 1$ for all $h,h' \in V_H$.

Conversely, let $M$ be a $\psd$-isomorphism matrix and let $v_{(g,h)}$ be the vectors in a Gram decomposition. For any $g\in V_G$ set $v_g=\sum_h v_{(g,h)}$ and for any $h\in V_H$ define $v_h=\sum_g v_{(g,h)}$. 
As before, we can easily show that  $v_g=v_{g'}$ for all $g,g'\in V_G$ and $v_h=v_{h'}$ for all $h,h'\in V_H$, and thus we use $v_G$ and $v_H$ to refer to these two vectors. But we also have that
$$|V_G|v_{G}=\sum_g \sum_h  v_{(g,h)}= \sum_h \sum_g  v_{(g,h)}=|V_H|v_{H},$$
which in turn implies that $v_G=v_H=:v$ (as $\psd$-isomorphic graphs have the same number of vertices). Lastly, extend $M$ be adding an extra row/column indexed by the vector $v$. It is then straightforward to check that the augmented matrix    satisfies \eqref{lasfirst}-\eqref{laslast}. 
\end{proof}

Combining several results of this paper, we have the following:

\begin{theorem}\label{thm:dnnsummary}
Let $G$ and $H$ be graphs. Then the following are equivalent:
\begin{enumerate}
\item[$(1)$] $G \cong_\psd H$.
\item[$(2)$] $G$ and $H$ are equivalent.
\item[$(3)$] $\vartheta'(G \diamond H) = |V_G| = |V_H|$.
\item[$(4)$] The first level of the Lasserre hierarchy for isomorphism of $G$ and $H$ has a nonnegative solution.
\end{enumerate}
\end{theorem}

\begin{remark}\label{rem:dnnsummary}
As mentioned in Section~\ref{subsec:cohalggraphs}, graphs $G$ and $H$ are equivalent if and only if they are not distinguished by the Weisfeiler-Leman method. In turn, the latter is known to have many equivalent formulations, e.g., in terms of logic and pebbling games on graphs~\cite{CFI}. However, the connections to Schrijver's theta function and the first level of the Lasserre hierarchy given in the above theorem appear to be new.
\end{remark}

\appendix

\section{Tools from algebra}

For our characterizations of $\dnn$- and $\psd$-isomorphism, we will need some tools from algebra. The first result was used in~\cite{friedland} and says that certain isomorphisms between matrix algebras can be realized as conjugation by a unitary matrix. 

\begin{lemma}\label{lem:unitary}
If $\A$ and $\B$ are self-adjoint unital subalgebras of $\C^{n \times n}$ and $\phi: \A \to \B$ is a trace-preserving isomorphism such that $\phi(X^\dagger) = \phi(X)^\dagger$ for all $X \in \A$, then there exists a unitary matrix $U \in \C^{n \times n}$ such~that
\[\phi(X) = UXU^\dagger \text{ for all } X \in \A.\]
\end{lemma}

\begin{lemma}\label{lem:01basis}
Suppose that $G$ is a graph with partially coherent algebra $\PA_G$, and let $\mathcal{D} = \{I \schur X : X \in \PA\}$. Then, $\mathcal{D}$ is a subalgebra  of $\PA_G$ and there exists an orthogonal basis of $\mathcal{D}$ consisting of diagonal 01 matrices.
\end{lemma} 
\begin{proof}

By definition  $\PA_G$ is closed under Schur product with $I$, and thus  $\mathcal{D} \subseteq \PA$. Since $\PA$ is a vector space, and $X \mapsto I \schur X$ is the projection onto the vector space of diagonal matrices, we have that $\mathcal{D}$ is a subspace of $\PA_G$, and that $I \in \mathcal{D}$. Also, since $\PA_G$ is closed under matrix product and the product of any two diagonal matrices is diagonal, we have that $\mathcal{D}$ is closed under matrix product. 

For any matrix  $D \in \mathcal{D}$  there exist distinct $\alpha_1, \ldots, \alpha_k \in \C$ such that $D = \sum_i \alpha_i I_i$ where $I_1, \ldots, I_k$ are~01 diagonal matrices with distinct nonzero entries. It remains to show  that all  matrices $I_i$ lie in $\mathcal{D}$. For this, note that for any  $i \in [k]$ we have that 
\[I_i = \prod_{j \ne i} \frac{1}{\alpha_i - \alpha_j}(D - \alpha_j I),\]
which shows that $I_i\in \mathcal{D}$ as $\mathcal{D}$ is  a subalgebra of $\PA_G$ with  $I,D\in \mathcal{D}$. From here it is easy to see that there exists a set of 01 diagonal matrices of $\mathcal{D}$ whose nonzero entries are disjoint and who span $\mathcal{D}$. This is our desired basis of $\mathcal{D}$.
\end{proof} 

%

In order to use Lemma~\ref{lem:unitary}, we need to prove that equivalences and partial equivalences are trace-preserving. This was done for equivalences in~\cite{friedland}, and their proof can be used to prove the same for partial equivalences, which we do here.

\begin{lemma}\label{pequiv_trace_pres}
Suppose that $G$ and $H$ are graphs with partially coherent algebras $\PA_G$ and $\PA_H$ respectively. If $\phi: \PA_G \to \PA_H$ is a partial equivalence of $G$ and $H$, then $\phi$ is trace-preserving.
\end{lemma}
\begin{proof}
For  any $X \in \PA_G$ and we have that $ \tr(\phi(X)) = \tr(I \schur \phi(X)) = \tr(\phi(I \schur X))$. 
Thus it suffices to show that $\phi$ is trace-preserving on the  subalgebra  $\mathcal{D} = \{I \schur X : X \in \PA_G\}$. By Lemma~\ref{lem:01basis}, there exist  diagonal 01 matrices $I_1, \ldots, I_d$ that form an orthogonal basis of $\mathcal{D}$. By linearity, it suffices to show that  $\phi$ preserves the trace of each individual  $I_i$.

First, since $\phi$ is a partial equivalence, we have by definition that
\[\phi(I_i) = \phi(I \schur I_i) = I \schur \phi(I_i),\]  
and therefore we have that $\phi(I_i)$ is diagonal for all $i$. Moreover, since $I_i^2 = I_i$ for all $i \in [d]$, we have that
\[\phi(I_i)^2 = \phi(I_i^2) = \phi(I_i),\]
and therefore $\phi(I_i)$ is a 01 diagonal matrix for all $i \in [d]$. Let $n_i = \tr(I_i)$ be the number of 1's in $I_i$, and let $n'_i = \tr(\phi(I_i))$ be the number of 1's in $\phi(I_i)$. We aim to show that $n'_i = n_i$.

Recall that $J \in \PA_G$, and thus $J_i := I_i J I_i \in \PA_G$ for all $i \in [d]$. Let $J'_i := \phi(J_i) = \phi(I_i)J\phi(I_i)$. It is easy to see that  $J_i^2 = n_i J_i$, and similarly $(J'_i)^2 = n'_i J'_i$. Therefore,
\[\phi(J_i)^2 = \phi(J_i^2) = \phi(n_i J_i) = n_i \phi(J_i).\]
However, we also have that
\[\phi(J_i)^2 = (J'_i)^2 = n'_i J'_i = n'_i \phi(J_i).\]
Of course this implies that $n'_i = n_i$ and we are done.
\end{proof}

Lastly, we will need to use the fact that the orthogonal projection onto a unital self-adjoint algebra is a completely positive map, see \cite[Theorems 1.5.10 and 1.5.11]{ozawa}.

\begin{lemma}\label{lem:algproj}
Let $\A$ be a self-adjoint subalgebra of $\C^{n \times n}$ containing the identity. If $\Pi$ is the orthogonal projection onto $\A$, then $\Pi$ is a CPTP unital map.
\end{lemma}

\bibliographystyle{plainurl}
\bibliography{MaxEnt}

\end{document}